\documentclass[10pt,a4paper]{amsart}
\usepackage[T1]{fontenc}
\usepackage[utf8]{inputenc}
\usepackage[british]{babel}
\usepackage[a4paper,margin=1in]{geometry}
\usepackage{microtype}
\usepackage{amsmath,amssymb,amsthm,mathtools,mathrsfs}
\IfFileExists{newtxtext.sty}{\usepackage{newtxtext,newtxmath}}{\usepackage{mathptmx}}
\usepackage{tikz-cd,needspace}
\usepackage[shortlabels]{enumitem}
\usepackage[hidelinks]{hyperref}

\numberwithin{equation}{section}
\theoremstyle{plain}
\newtheorem{theorem}{Theorem}[section]
\newtheorem{proposition}[theorem]{Proposition}
\newtheorem{lemma}[theorem]{Lemma}
\newtheorem{corollary}[theorem]{Corollary}
\newtheorem*{theorema}{Theorem A}
\newtheorem*{theoremb}{Theorem B}
\newtheorem*{theoremc}{Theorem C}
\newtheorem*{theoremd}{Theorem D}
\newtheorem*{theoreme}{Theorem E}
\theoremstyle{definition}
\newtheorem{definition}[theorem]{Definition}
\newtheorem{example}[theorem]{Example}
\newcommand{\C}{\mathbb C}
\newcommand{\R}{\mathbb R}
\newcommand{\Z}{\mathbb Z}
\newcommand{\N}{\mathbb N}
\newcommand{\Q}{\mathbb Q}
\newcommand{\Span}{\operatorname{Span}}
\newcommand{\rk}{\operatorname{rk}}

\newcommand{\Spec}{\operatorname{Spec}}

\newcommand{\Hom}{\operatorname{Hom}}
\newcommand{\Aut}{\operatorname{Aut}}
\newcommand{\supp}{\operatorname{supp}}
\newcommand{\Cone}{\operatorname{Cone}}
\newcommand{\cO}{\mathcal O}
\newcommand{\cF}{\mathscr F}
\newcommand{\cG}{\mathscr G}
\newcommand{\g}{\mathfrak g}
\newcommand{\h}{\mathfrak h}

\author{Maur\'icio Corr\^ea}
\address{Maur\'icio Corr\^ea\\
Universit\`a degli Studi di Bari,
Via E. Orabona 4, I-70125, Bari, Italy}
\email[M. Corr\^ea]{mauricio.barros@uniba.it, mauriciomatufmg@gmail.com}

\author{Jos\'e Seade}
\address{Jos\'e Seade\\
Instituto de Matem\'aticas-Cuernavaca, Universidad Nacional Aut\'onoma de M\'exico.}
\email{jseade@im.unam.mx}

\date{19 October 2026}
\subjclass[2020]{Primary 34A05; Secondary 34M35, 32M25, 34C20, 37F75, 32S65, 14L24, 14M25, 57R30}
\keywords{holomorphic first integrals, integrability, Poincar\'e--Dulac normal form, singular holomorphic foliations, transverse analytic envelope, monodromy, Baum--Bott residues}

\title[Transverse Envelopes and Baum--Bott Rationality]{Transverse Analytic Envelopes and the Baum--Bott Rationality Conjecture}
\hypersetup{
  pdftitle={Transverse Analytic Envelopes and the Baum--Bott Rationality Conjecture},
  pdfauthor={Maur\'icio Corr\^ea and Jos\'e Seade},
  pdfsubject={Transverse analytic envelopes, semiglobal integrability, holonomy, reconstruction and Baum--Bott rationality},
  pdfkeywords={holomorphic first integrals, integrability, Poincar\'e--Dulac normal form, singular holomorphic foliations, transverse analytic envelope, monodromy, Baum--Bott residues}
}

\begin{document}

\begin{abstract}
We study holomorphic first integrals near singular leaves through
transverse analytic envelopes. For diagonal actions these are normal toric
quotients, and complete semiglobal integrability along singular orbits with
independent support weights is equivalent to maximal resonance rank and finite
monodromy. For commuting fields in convergent simultaneous Poincar\'e--Dulac
normal form, induced derivations determine the first-integral algebra.
Labelled quotients and logarithmic residues reconstruct weights at maximal
rank; Grassmannians describe the ambiguity otherwise. Polynomial examples
show losses invisible to prescribed finite jets. For the Baum--Bott
rationality conjecture, we prove rationality and an explicit residue formula
for smooth compact singular components with product charts and a common
invertible semisimple transverse linear part of maximal resonance rank.
Equivariant cohomology gives higher-degree vanishing for affine diagonal
torus actions, including intersecting singular strata.
\end{abstract}
\maketitle

\section*{Introduction}
A holomorphic first integral is constant along the leaves of a foliation.
Near a regular point, the Frobenius theorem supplies independent first
integrals as transverse coordinates: locally, the leaves are the fibres
of a submersion. Near a singular leaf, first integrals may fail to exist,
and those defined on a transversal need not extend along the leaf.

Three obstructions arise even for diagonal linear vector fields.
Distinct leaves may have intersecting closures, so that the leaf space
is not $T_1$. Integral resonances among the eigenvalues may detect only
part of the transverse dynamics. Finally, monodromy along a singular
leaf may prevent a local integral from extending as a single-valued
function. Nonlinear terms introduce a further obstruction, which may
be invisible to any prescribed finite jet. The algebras of local and
semiglobal first integrals distinguish these phenomena.

These obstructions are transverse in nature. In Haefliger's description
of regular foliations \cite{Haefliger58,Haefliger62,Haefliger84},
transverse structures are expressed by pseudogroup cocycles or holonomy
groupoids, in place of the generally ill-behaved space of leaves.
Transversely holomorphic structures and their deformations belong to this
setting \cite{GHS}. A transversal to a regular foliation is a local
model for its leaf space; at a singular leaf, however, the transversal
carries a residual foliation.

The \emph{transverse analytic envelope} is an analytic germ through
which every holomorphic map constant on the transverse leaves factors
uniquely. When it exists, this categorical quotient represents the algebra
of holomorphic first integrals. It is the usual transversal at a regular
leaf. At a singular leaf, it retains the holomorphic functions detected
by the transverse foliation, even when the leaf space is non-separated.
The transverse pseudogroup descends to the envelope and acts on these
functions.

For diagonal actions, these envelopes are normal toric varieties.
Transport along a singular orbit defines a flat logarithmic connection,
whose monodromy-invariant monomials give the semiglobal first integrals
when the support weights are independent. For commuting fields in
convergent simultaneous Poincar\'e--Dulac normal form, derivations on the
semisimple envelope determine which integrals survive the nonlinear
terms. Monodromy then governs their extension along the leaf.

The labelled quotient and the residues of its logarithmic connection
also reconstruct the ambient weights up to linear equivalence at maximal
resonance rank; at smaller rank, inequivalent actions may have the same
envelope. These results lead to applications to the \emph{Baum--Bott
rationality conjecture}, which concerns rational characteristic polynomials
of degrees $q+1<d<n$ and compact connected singular components of
codimension-$q$ foliations. For smooth compact components with local
product charts, maximal resonance rank makes the transverse eigenvalues
commensurable and yields an explicit residue formula. For affine diagonal
torus actions, equivariant cohomology proves higher-degree vanishing,
including at intersections of singular strata. The latter statement
concerns a generally noncompact singular set.

\medskip

We write $\N=\Z_{\ge0}$. For the diagonal case, let $\g=\C^k$ act on $\C^n$
with spanning weights
$\Lambda_1,\ldots,\Lambda_n\in\g^*$ by
$t\cdot z=(e^{\Lambda_j(t)}z_j)_j$. For a point with coordinate support
$I=\{i:z_i\ne0\}$, put
$J=I^c$, $r_I=\dim\Span\{\Lambda_i:i\in I\}$ and
$\h_I=\bigcap_{i\in I}\ker\Lambda_i$. The transverse foliation is
the product of the foliation by points on $\C^{|I|-r_I}$ and the diagonal
$\h_I$-action with weights $\mu_j=\Lambda_j|_{\h_I}$. Its toric factor is
$$
Q_I=\Spec\C[S_I],\qquad
S_I=\{\alpha\in\N^J:\textstyle\sum_j\alpha_j\mu_j=0\}.
$$
The quotient represents holomorphic first integrals. Transport along
the singular orbit supplies their monodromy.

Suppose $I$ is nonempty, singular ($r_I<k$) and independent
($r_I=|I|$), so its orbit is
$L_I=(\C^*)^I$. Write $q=n-k$ and $M_I=\operatorname{gp}(S_I)$.
The resonance semigroup has \emph{maximal rank} when
$\rk M_I=q$, equivalently $\dim Q_I=q$. Define
the additive coefficients $C_I(\alpha)=(c_i(\alpha))_{i\in I}$ by
$$
\sum_{j\in J}\alpha_j\Lambda_j
=\sum_{i\in I}c_i(\alpha)\Lambda_i.
$$
Transport gives the logarithmic connection
$\nabla_I=d-\sum_i(dz_i/z_i)\delta_i$, where
$\delta_i(X^\alpha)=c_i(\alpha)X^\alpha$; monodromy multiplies a
character by $e^{2\pi i c_i(\alpha)}$ around the $i$th coordinate loop.

\begin{theorema}
The algebra of holomorphic first-integral germs along the entire leaf
$L_I$ is the analytic local algebra of
$Q_I^{\mathrm{sg}}=\Spec\C[S_I^{\mathrm{inv}}]$, where
$S_I^{\mathrm{inv}}=\{\alpha\in S_I:C_I(\alpha)\in\Z^I\}$.
The character $X^\alpha$ corresponds to
$z^{-C_I(\alpha)}w^\alpha$. There are $q$ functionally independent
semiglobal holomorphic first integrals if and only if
$\dim Q_I=q$ and envelope monodromy is finite.
\end{theorema}

A germ along a leaf is represented on a neighbourhood of the entire
leaf; the transverse radius may vary. The number of independent integrals
is $\rk\operatorname{gp}(S_I^{\mathrm{inv}})$, which may be smaller than
the rank of the monodromy-invariant character lattice: an invariant
meromorphic function need not extend holomorphically at the distinguished
point. Section~\ref{sec:integrability} proves the theorem and gives examples.

\subsection*{Nonlinear reduction}
On $(\C^s,0)$ consider commuting fields $X_a=S_a+R_a$, $1\le a\le d$,
where the $S_a$ are independent diagonal linear fields, the coefficients
of $R_a$ vanish to order at least two, and $[S_a,R_b]=0$ for all $a,b$.
We assume that these simultaneous normal-form coordinates are holomorphic.
Let $q_0:(\C^s,0)\to(Q,o)$ be the semisimple toric envelope.

\begin{theoremb}
Each $R_a$ induces a commuting holomorphic derivation $\overline R_a$
on $Q$, and
$$
\cO_{\C^s,0}^{X_1,\ldots,X_d}
=q_0^*\bigl(\bigcap_a\ker\overline R_a\bigr).
$$
There are $s-d$ independent holomorphic first integrals if and only if
$\dim Q=s-d$ and every $\overline R_a$ vanishes. If these transverse
models have locally constant transverse transition maps along a leaf,
fixing zero, preserving both systems and inducing holonomy in the dense torus of $Q$,
then complete semiglobal integrability holds
precisely when these two conditions hold and envelope holonomy is finite.
\end{theoremb}

For any prescribed jet order, polynomial examples in
Section~\ref{sec:nonlinear} have the same jet as an integrable
semisimple system but fewer first integrals. A suspension gives
successive losses under nonlinear reduction and monodromy: three
independent integrals become two, and then one.

\medskip
The inverse problem requires coordinate labels: an abstract toric germ
does not specify which source monomials it represents. Fix a spanning
residual configuration of dimension $d$, its labelled semigroup
$S\subset\N^s$, and set $M=\operatorname{gp}(S)$,
$R=M\otimes_\Z\C$, $\delta=s-d-\rk M$.

\begin{theoremc}
Labelled residual configurations with semigroup $S$, modulo linear
equivalence, form a dense $G_\delta$ subset of
$\operatorname{Gr}_{\delta}(\C^s/R)$; this Grassmannian has
dimension $\delta d$. Thus the labelled quotient map determines the
residual configuration exactly when $\delta=0$. At a singular
coordinate stratum indexed by $I\ne\varnothing$ with
$\dim Q_I=n-k$, the labelled quotient map and logarithmic residues
determine all ambient weights up to linear equivalence.
\end{theoremc}

The reconstruction is explicit (Section~\ref{sec:inverse}). When
$\dim Q_I<n-k$, the envelope, its connection and its monodromy may
fail to determine the transverse Camacho--Sad indices. Conversely,
while keeping the residual action fixed, one may prescribe arbitrary
toric monodromy within Poincar\'e-domain configurations in linear
general position (Section~\ref{sec:realisation}). Classical residues
across different supports are compared in Section~\ref{sec:residues}.

\medskip

\subsection*{Applications to Baum--Bott residues}
Baum and Bott introduced residues as localised characteristic classes
of the normal sheaf of a singular holomorphic foliation and computed
them by transverse residues along components of the expected dimension
\cite{BaumBott}. Their rationality conjecture concerns the distinguished
residues of rational characteristic polynomials in degrees
$q+1<d<n$ for codimension-$q$ foliations
\cite[Conjecture~2.5]{BrasseletSuwa}. Suwa established rationality
for certain complete-intersection-type singularities \cite{Suwa1984}.
Brasselet and Suwa later compared Baum--Bott residues with Nash
residues, proved that their difference is rational, and obtained
further cases of the conjecture \cite{BrasseletSuwa}.

Bracci and Suwa proved continuity of Baum--Bott residues under
suitable deformations \cite{BracciSuwa}. More recently, Corr\^ea and
Suwa established a pull-back formula for distinguished localised
classes under compatibility conditions on the normal sheaves;
rationality is preserved by these pull-backs and, for proper
surjective generically finite maps of equal dimension, is equivalent
to rationality of the pulled-back residue \cite{CorreaSuwa}.
For the smooth compact components considered here, maximal resonance
rank makes the common transverse eigenvalues proportional to integers.
Comparison with the normal model then gives an explicit formula in
terms of rational characteristic classes of the normal bundle.

\begin{theoremd}
Let $B$ be a smooth compact connected singular component of a
saturated codimension-$q$ holomorphic foliation, $q\ge1$.
Suppose the foliation has local product charts with holomorphic
transverse fields on $(\C^{q+1},0)$ whose common linear part is
invertible and semisimple, with toric envelope of dimension $q$.
Its distinguished Baum--Bott residues are rational for every rational
characteristic polynomial of degree $q+1\le d\le n$.
\end{theoremd}

Theorem~\ref{thm:bb-rationality} gives the residue as a rational
multiple of a characteristic class of the normal bundle, after a finite
\'etale cover. The formula continues to hold for nonlinear transverse
generators, including those that reduce the number of local first
integrals (Corollary~\ref{cor:bb-nonlinear-generators}).
Example~\ref{ex:bb-no-retraction} computes a nonzero residue in a
neighbourhood that admits no holomorphic retraction onto the singular
component. The same formula also covers $d=q+1$ and $d=n$, the two
boundary degrees outside the range quoted above.

\begin{theoreme}
Let an integral diagonal algebraic torus action on $\C^N$ have a
saturated orbit foliation of codimension $q\ge1$. Every distinguished
Baum--Bott class in degree $q+1<d\le N$ vanishes. In degree $d=q+1$
the class is rational for rational characteristic polynomials.
\end{theoreme}

The singular set of a saturated affine orbit foliation is generally
noncompact. The proof includes intersections of singular strata through
the equivariant cohomology of a coordinate subspace arrangement. The reconstruction theorem connects
it to the original diagonal actions of complex vector groups: maximal
envelope rank and finite monodromy along an independent positive-dimensional
singular orbit force an equivalent rational weight configuration
(Corollary~\ref{cor:bb-reconstructed-vanishing}). Both applications are
proved in Section~\ref{sec:bb-application}.

Holomorphic first integrals and singular
leaf spaces have a classical history \cite{MM,Morel,LeonScardua}. The
singular transverse viewpoint used here follows \cite{ACSV}; related
holonomy and neighbourhood constructions for singular foliations appear
in \cite{AS,AZ,LGR-neighbourhood}. Our toric
calculations use \cite{Fulton} and analytic invariant theory
\cite{Snow,HH,GreMie}; exponential actions and vector configurations are
also studied in \cite{Panov}. The nonlinear reduction belongs to the
interaction between normal forms and integrability
\cite{StolovitchSCI,ZungNormalForms,Stolovitch}; related work on
resonances and analytic first integrals of differential systems appears
in \cite{ChenYiZhang,ZhangFirstIntegrals}. The residue comparisons use
\cite{BaumBott,CamachoSad,BSS,LehmannSuwa,CamachoLehmann}.

\medskip
Sections~\ref{sec:local}--\ref{sec:integrability} develop the
envelopes, transport and semiglobal first integrals.
Sections~\ref{sec:realisation}--\ref{sec:residues} treat realisation,
reconstruction and residues. Section~\ref{sec:nonlinear} considers
nonlinear normal forms. Section~\ref{sec:bb-application} applies the
theory to Baum--Bott residues on compact components and on affine space.

\section{Local envelopes and restriction maps}\label{sec:local}
Let $\Lambda_1,\ldots,\Lambda_n$ span $\g^*$, $\g=\C^k$, and let
$\cF$ be the orbit foliation of $t\cdot z=(e^{\Lambda_j(t)}z_j)_j$.
Its generators are $X_v=\sum_j\Lambda_j(v)z_j\partial_{z_j}$.
We use the Stefan--Sussmann orbit foliation, including its singular
leaves. The coordinate stratum indexed by $I$ consists of points with
support exactly $I$; it is singular when $r_I<k$ and has independent
support weights when $r_I=|I|$. For $p$ in this stratum, use
$J,r_I,\h_I$ as in the introduction.
A configuration is in \emph{linear general position} if every subcollection
of at most $k$ weights is independent. It is in the \emph{Poincar\'e
domain} if some $\xi\in\g$ satisfies
$\operatorname{Re}\Lambda_j(\xi)>0$ for all $j$, equivalently if zero
does not belong to their real convex hull. These two conditions are
independent.

\begin{proposition}\label{prop:residual}
The leaf through $p$ has dimension $r_I$ and infinitesimal stabiliser
$\h_I$. Choose $B\subset I$ such that the $\Lambda_b$, $b\in B$, form
a basis of the support span. The slice $\Sigma_B=\{z_b=p_b:b\in B\}$
is transverse at $p$, and its induced foliation is the product of the
foliation by points on $\C^{|I|-r_I}$ with the diagonal $\h_I$-foliation on
$\C^J$ of weights $\mu_j=\Lambda_j|_{\h_I}$. In linear general position,
singular coordinate strata satisfy $|I|=r_I<k$.
\end{proposition}
\begin{proof}
The evaluation map $v\mapsto X_v(p)$ has kernel $\h_I$, giving the
dimension and transversality. Split $\g=V\oplus\h_I$. Tangency to
$z_b=p_b$ forces the coefficients of the $V$-generators to vanish on
the slice, since the matrix $(\Lambda_b|_V)$ is invertible and every
$p_b$ is nonzero. The remaining fields vanish in the support coordinates
and equal $\sum_{j\in J}\mu_j(v)w_j\partial_{w_j}$ in the others.
\end{proof}

\subsection{The analytic invariant algebra}
\begin{definition}\label{def:envelope}
Let $(\Sigma,\cF,p)$ be a germ of holomorphic foliated space. A
holomorphic map germ $q:(\Sigma,p)\to(Q,o)$ to a complex analytic germ,
constant on the local leaves, is an \emph{analytic envelope}
of $\cF$ at $p$ if every holomorphic map germ $(\Sigma,p)\to(Y,y)$
constant on the local leaves factors uniquely as $g\circ q$ with
$g:(Q,o)\to(Y,y)$ holomorphic.
\end{definition}

This is the categorical quotient of the foliated germ, when it exists.
On a transversal we call it the transverse analytic envelope.
The factorisation property determines it up to unique isomorphism.
For the foliation by points it is the identity of $\Sigma$.
In the diagonal models below, two points have the same image exactly
when all holomorphic first integrals agree on them.

Consider residual weights $\mu_1,\ldots,\mu_s\in\h^*$. Put
$$
H=\{(e^{\mu_j(v)})_j:v\in\h\},\quad
G=\overline H^{\mathrm{Zar}}\subset(\C^*)^s,\quad
L=\{a\in\Z^s:\textstyle\sum_j a_j\mu_j=0\},\quad S=L\cap\N^s.
$$
The semigroup $S$ indexes the invariant monomials and is called the
\emph{resonance semigroup}. Its group completion is the subgroup
$\operatorname{gp}(S)=S-S$ of $\Z^s$; it need not equal $L$.
The connected algebraic subgroup $G$ is a torus, and its character
lattice is $\Z^s/L$: a character is trivial on $H$, hence on $G$,
exactly when its differential vanishes on $\h$.

\begin{theorem}\label{thm:envelope}
The semigroup $S$ is finitely generated and saturated in
$M=\operatorname{gp}(S)$, and $Q=\Spec\C[S]$ is a normal affine
toric variety. The map $q:\C^s\to Q$ is the affine categorical quotient
by $G$ and, analytically, its analytic Hilbert quotient.
At $o=q(0)$ it satisfies
$$
q^*\cO_{Q,o}=\cO_{\C^s,0}^{\cF}
=\left\{\sum_{\alpha\in S}a_\alpha w^\alpha:
                 \text{the series converges near }0\right\}.
$$
Every leaf-constant holomorphic map germ factors uniquely through $q$.
Moreover $\dim Q=\rk M$, and $o$ is smooth exactly when
$S\simeq\N^{\rk M}$. Globally, two points have the same quotient
image exactly when their $G$-orbit closures intersect; each fibre
contains a unique closed $G$-orbit.
\end{theorem}
\begin{proof}
The lattice $L$ is saturated: if $ma\in L$ for some positive
integer $m$, then
$m\sum_j a_j\mu_j=0$, and the complex vector space $\h^*$ has
no additive torsion, so $a\in L$. The rational polyhedral cone
$L_{\R}\cap\R_{\ge0}^s$ has semigroup of lattice points $S$.
Gordan's lemma gives a finite Hilbert basis
$\beta^1,\ldots,\beta^N$; the coordinatewise non-negativity
argument shows that $S$ is saturated in $\operatorname{gp}(S)$.
Consequently $\C[S]$ is the integrally closed coordinate ring of a
normal affine toric variety, of dimension $\rk M$. The smoothness
criterion at its distinguished fixed point is the usual criterion for
an affine toric cone: $S$ must be freely generated as a monoid by
$\dim Q$ elements~\cite{Fulton}.

A convergent power series $f(w)=\sum_{\alpha}a_\alpha w^\alpha$
is a first integral precisely when every residual field annihilates
it. Termwise differentiation gives
\[
Y_vf=\sum_\alpha a_\alpha
\left(\sum_j\alpha_j\mu_j(v)\right)w^\alpha.
\]
Linear independence of the monomials as analytic germs therefore
shows that its nonzero coefficients can occur only for
$\alpha\in S$. This also proves equality with the local analytic
invariants of the Zariski-closed torus $G$: each such monomial
corresponds to a character trivial on the exponential image and
hence on its Zariski closure.

If $S=\{0\}$, then all first integrals are constant and the claims
follow. Assume henceforth that $S\ne\{0\}$.

It remains to justify analytically the descent of a convergent
series, rather than merely its formal descent. Choose, for every
$\alpha\in S$, a representation
$\alpha=\sum_{a=1}^Nu_a(\alpha)\beta^a$ with nonnegative
integral coefficients. Set $B=\max_a|\beta^a|$. Convergence of
$f$ on a polydisc gives $|a_\alpha|\le A R^{-|\alpha|}$,
where $A,R>0$ are fixed. Since
$|\alpha|\le B|u(\alpha)|$, for sufficiently small $\rho>0$
we obtain
\[
\sum_{\alpha\in S}|a_\alpha|\rho^{|u(\alpha)|}
\le A\sum_{\alpha\in\N^s}
\left(\frac{\rho^{1/B}}R\right)^{|\alpha|}<\infty.
\]
The power series
$g(x)=\sum_{\alpha\in S}a_\alpha x^{u(\alpha)}$ therefore
converges in a neighbourhood of the origin of $\C^N$. On the
subvariety $Q$, expressions of the same exponent in the Hilbert
basis define the same character $X^\alpha$. Hence $g|_Q$ is a
well-defined analytic germ with $q^*g=f$.

The pullback on local rings is injective. Indeed, the inclusion
$M\hookrightarrow\Z^s$ induces a surjective homomorphism of dense
tori $(\C^*)^s\to T_Q$, and this morphism is a submersion. If a
germ on $Q$ pulls back to zero near the origin of $\C^s$, then it
vanishes on a nonempty open subset of the dense torus of $Q$.
Irreducibility and the identity principle imply that the germ is
zero. This gives the required equality of analytic algebras.

For the universal property let
$F:(\C^s,0)\to(Y,y)$ be constant on local leaves, and embed the
target germ in a polydisc $\Delta^m$. Each coordinate function of
$F$ descends uniquely to a germ $g_i$ on $Q$. If $h$ is a local
equation defining $Y$ in $\Delta^m$, then
$q^*(h(g_1,\ldots,g_m))=h(F)=0$. Injectivity of $q^*$ forces
$h(g_1,\ldots,g_m)=0$ on $Q$, and hence the descended tuple
takes values in $Y$. The same injectivity proves
uniqueness. Finally, reductive affine quotient theory identifies
equality of quotient values with intersection of orbit closures and
gives the unique closed orbit in each fibre~\cite{HH,GreMie}.
\end{proof}

Consequently the full transverse envelope at support $I$ is
$\C^{|I|-r_I}\times Q_I$. Write $q_I:\C^J\to Q_I$ for the
residual quotient and $o_I=q_I(0)$. The same argument, with the free coordinates
as additional weight-zero variables, proves its universal property.

\begin{proposition}\label{prop:non-separation}
For a diagonal residual foliation at a fixed point, the local leaf space
is $T_1$ if and only if all residual weights vanish. If some residual
weight is nonzero, then no continuous map to a $T_1$ space has the connected
local leaves as its fibres. In particular, this obstruction occurs at
every singular coordinate stratum of a spanning ambient diagonal action.
\end{proposition}
\begin{proof}
On a singular coordinate stratum $r_I<k$, and at least one residual
weight is nonzero. Indeed, if every
$\Lambda_j|_{\h_I}$ vanished for $j\notin I$, then all ambient
weights would annihilate $\h_I$, contradicting $\h_I\ne0$.
Choose a residual coordinate $w_j$ with $\mu_j\ne0$ and a vector
$v\in\h_I$ with $\mu_j(v)=1$. The orbit through a sufficiently small
$a e_j$, $a\ne0$, contains $e^{-t}a e_j$ for every real $t\ge0$,
and these points tend to the origin. In a polydisc they belong to
the same connected local leaf. Any continuous map to a $T_1$
space constant on such leaves takes the same value at $ae_j$ and
at the origin, although these points belong to distinct leaves.
Therefore it cannot have the local leaves as its fibres.
Conversely, if every residual weight vanishes, then the foliation is the
foliation by points and the identity of the transversal is a separated
geometric quotient.
\end{proof}

The Poincar\'e-domain condition forces $S=\{0\}$: evaluate a nonnegative relation on
the separating direction and take real parts. In particular the
envelope at the origin of an ambient Poincar\'e-domain action is a
point. The converse fails, since absence of integral resonances does
not exclude real convex relations.

\subsection{Fibres and coordinate strata}
Put $M_G=X^*(G)=\Z^s/L$, and write $\lambda_j\in M_G$ for the
coordinate characters. For
$K\subset\{1,\ldots,s\}$ set
$$
C_K=\Cone\{\lambda_j:j\in K\},\qquad
\ell_K=C_K\cap(-C_K),\qquad B_K=\{j\in K:\lambda_j\in\ell_K\}.
$$
For $x$ of support $K$, let $x^\circ$ retain the coordinates in $B_K$
and set the others to zero.

\begin{lemma}\label{lem:BK-resonance}
For $j\in K$, the following are equivalent:
\begin{enumerate}[(i)]
\item $j\in B_K$;
\item there exists $\alpha\in S$ with $\supp(\alpha)\subset K$ and $\alpha_j>0$.
\end{enumerate}
Moreover
\begin{equation}\label{eq:cone-BK}
\Cone_{\R_{\ge0}}\{\lambda_j:j\in B_K\}=\ell_K.
\end{equation}
\end{lemma}

\begin{proof}
If $j\in B_K$, then $-\lambda_j\in C_K$. Rationality of $C_K$
gives $-\lambda_j=\sum_{i\in K}c_i\lambda_i$ with
$c_i\in\Q_{\ge0}$. Clearing denominators gives a nonnegative
integral relation supported in $K$, with positive $j$th coefficient.
Its exponent lies in $L\cap\N^s=S$.
Conversely, a relation with $\alpha_j>0$ gives
\[
-\lambda_j=\sum_{i\in K\setminus\{j\}}
\frac{\alpha_i}{\alpha_j}\lambda_i\in C_K,
\]
so $j\in B_K$.

To prove \eqref{eq:cone-BK}, pass to the pointed rational cone
$C_K/\ell_K$. Choose an integral functional strictly positive
on its nonzero generators, and lift it to
$\nu\in\Hom(M_G,\Z)$. Then $\nu$ vanishes on $\ell_K$ and
is positive on $\lambda_j$ for $j\notin B_K$.
For $u\in\ell_K$, write $u=\sum_{j\in K}c_j\lambda_j$ with
$c_j\ge0$. Applying $\nu$ forces $c_j=0$ outside $B_K$.
Thus $\ell_K$ is generated by the weights indexed by $B_K$.
\end{proof}

\begin{lemma}\label{lem:closed-orbit}
The orbit $G\cdot x$ is closed in $\C^s$ if and only if $B_K=K$, equivalently if and only if $C_K$ is a vector subspace of $(M_G)_\R$.
\end{lemma}

\begin{proof}
If $B_K=K$, then for every $j\in K$ Lemma~\ref{lem:BK-resonance}
gives an invariant monomial supported in $K$, nonzero at $x$, with
positive $j$th exponent. These monomials exclude any point of
$\overline{G\cdot x}$ with smaller support. In the coordinate torus
$(\C^*)^K$, the orbit is a coset of an algebraic subtorus and hence
closed. It is therefore closed in $\C^s$.

If $B_K\ne K$, then the functional $\nu$ constructed above defines
a one-parameter subgroup $\gamma_\nu:\C^*\to G$ with
\[
\bigl(\gamma_\nu(t)\cdot x\bigr)_j
=t^{\langle\lambda_j,\nu\rangle}x_j.
\]
The exponent is zero on $B_K$ and positive on $K\setminus B_K$.
As $t\to0$, the orbit approaches $x^\circ$, whose smaller support
excludes it from $G\cdot x$. Thus the orbit is not closed.
\end{proof}

The subset $B_K$ determines the closed orbit in each quotient fibre.
\begin{proposition}\label{prop:fibres}
Let $q:\C^s\to Q$ be the transverse analytic envelope.  For every $x\in\C^s$:
\begin{enumerate}[(i)]
\item $q(x)=q(x^\circ)$;
\item $G\cdot x^\circ$ is closed;
\item $G\cdot x^\circ$ is the unique closed $G$-orbit in the fibre $q^{-1}(q(x))$.
\end{enumerate}
Consequently, for all $x,y\in\C^s$,
\begin{equation}\label{eq:fibre-closed-orbit}
q(x)=q(y)
\Longleftrightarrow
G\cdot x^\circ=G\cdot y^\circ.
\end{equation}
\end{proposition}

\begin{proof}
Let $K=\supp(x)$.  It suffices for (i) to compare invariant monomials.  If $w^\alpha$ is invariant and nonzero at $x$, then $\supp(\alpha)\subset K$.  Every index occurring in $\alpha$ belongs to $B_K$ by Lemma~\ref{lem:BK-resonance}, so $(x^\circ)^\alpha=x^\alpha$.
If $w^\alpha(x)=0$, then $w^\alpha(x^\circ)=0$ as well.  Every invariant polynomial, and hence the quotient map, has the same value at $x$ and $x^\circ$.
Equation~\eqref{eq:cone-BK} shows that the cone generated by the weights in $B_K$ is the vector space
$\ell_K$.  Lemma~\ref{lem:closed-orbit} shows that $G\cdot x^\circ$ is closed.  The closed-orbit assertion of Theorem~\ref{thm:envelope} identifies
it with the unique closed orbit in $q^{-1}(q(x))$. Two fibres agree
precisely when their closed orbits agree, proving~\eqref{eq:fibre-closed-orbit}.
\end{proof}

\begin{proposition}
\label{prop:exp-closure}
Let $d_{\mathrm{eff}}=\dim\operatorname{Span}_{\C}\{\mu_j\}$ and
$L=\ker(\Z^s\to\h^*)$ as above. Then
$\dim H=d_{\mathrm{eff}}$ and $\dim G=s-\rk L$. The following
conditions are equivalent:
\begin{enumerate}[(i)]
\item the analytic exponential image $H$ equals $G$;
\item $\rk L=s-d_{\mathrm{eff}}$;
\item the complex relation space of the residual weights is spanned
by its integral points.
\end{enumerate}
\end{proposition}
\begin{proof}
The differential of the exponential homomorphism
$\h\to(\C^*)^s$ at the origin is the weight map
$v\mapsto(\mu_j(v))_j$, of rank $d_{\mathrm{eff}}$. Its image
is the Lie algebra of the connected immersed subgroup $H$. The
character lattice of $G$ is $\Z^s/L$, so its dimension is
$s-\rk L$. Since $H\subset G$, we have $\dim H\le\dim G$. Equality of these
dimensions means that the map $H\to G$ is locally surjective near
the identity. Its image is therefore an open subgroup of the connected
complex torus $G$, and consequently is all of $G$. Finally the
complex relation space has dimension $s-d_{\mathrm{eff}}$ and
contains $L\otimes_\Z\C$; these two spaces agree precisely when
their dimensions agree.
\end{proof}

For example, weights $1,\sqrt2$ give a one-dimensional
analytic exponential subgroup which is Zariski dense in
$(\C^*)^2$; its envelope is a point. In general a quotient by
$G$ identifies points lying on distinct $H$-orbits, as is also
clear from the closed-orbit description above.

For supports $I\subset I'$, put $D=I'\setminus I$, $J'=(I')^c$ and
$B_{I,I'}=S_I\cap\N^{J'}\subset S_{I'}$.

\begin{proposition}\label{prop:support}
Setting $w_d=0$, $d\in D$, defines
$$
\operatorname{res}_{I,I'}:\C[S_I]\longrightarrow\C[S_{I'}],\qquad
X^\alpha\longmapsto
\begin{cases}X^\alpha&\supp\alpha\subset J',\\0&\text{otherwise}.
\end{cases}
$$
Its kernel is the monomial ideal generated by exponents meeting $D$,
and its image is $\C[B_{I,I'}]$. The associated morphism
$\vartheta_{I,I'}:Q_{I'}\to Q_I$ factors through the closed toric
subvariety $\Spec\C[B_{I,I'}]\hookrightarrow Q_I$.
These morphisms compose for nested supports.
\end{proposition}
\begin{proof}
An $I$-resonance supported on $J'$ remains a resonance after restriction
from $\h_I$ to $\h_{I'}$. Nonnegative exponents make the displayed rule
multiplicative and identify its kernel and image. Successive coordinate
restrictions give the composition law.
\end{proof}

The factorisation of the restriction homomorphism is
\[
\C[S_I]\twoheadrightarrow\C[S_I\cap\N^{(I')^c}]
\lhook\joinrel\longrightarrow\C[S_{I'}].
\]
Thus the new monomial first integrals are indexed by
$S_{I'}\setminus(S_I\cap\N^{(I')^c})$, and
$\vartheta_{I,I'}$ is a closed immersion onto the indicated
toric subvariety exactly when these semigroups agree.

The envelope dimension may vary across coordinate strata. It is zero at
the origin of a Poincar\'e-domain action, while restricted weights
along a positive-dimensional singular orbit may acquire nonnegative
integral relations.

\Needspace{9\baselineskip}
\section{Transport, monodromy and the logarithmic connection}
\label{sec:transport}
\begin{proposition}\label{prop:pseudogroup-descent}
Let $(\Sigma_a,\cF_a,p_a)$ and $(\Sigma_b,\cF_b,p_b)$ be holomorphic
foliated germs with analytic envelopes
$q_a:(\Sigma_a,p_a)\to(Q_a,o_a)$ and
$q_b:(\Sigma_b,p_b)\to(Q_b,o_b)$. Let
$h:(\Sigma_a,p_a)\longrightarrow(\Sigma_b,p_b)$ be a germ of biholomorphism which maps local leaves of
$\cF_a$ bijectively onto local leaves of $\cF_b$.  Then there is a unique germ of biholomorphism $\bar
h:(Q_a,o_a)\longrightarrow(Q_b,o_b)$ making the square
\begin{equation}\label{eq:descent-square}
\begin{tikzcd}[column sep=4.0em,row sep=1.8em]
(\Sigma_a,\cF_a,p_a) \arrow[r,"h"] \arrow[d,"q_a"']
& (\Sigma_b,\cF_b,p_b) \arrow[d,"q_b"] \\
(Q_a,o_a) \arrow[r,"\bar h"']
& (Q_b,o_b)
\end{tikzcd}
\end{equation}
commute.
The construction respects identities, inverses and composition.  Every transverse pseudogroup whose elements preserve the induced foliation descends to a pseudogroup on the analytic envelope.
\end{proposition}

\begin{proof}
The leaf-constant map $q_b\circ h$ factors uniquely through $q_a$,
giving $\bar h$. Applying the same property to $h^{-1}$ gives its
inverse: both compositions factor the appropriate envelope map, so
uniqueness makes them identities. For composable germs $h_1,h_2$,
the maps $\overline{h_2\circ h_1}$ and $\bar h_2\circ\bar h_1$
factor the same leaf-constant map and therefore agree. Identities
descend by the same uniqueness.
\end{proof}

The kernel of the descended action is determined by the holomorphic first integrals.
\begin{corollary}\label{cor:effective-pseudogroup}
Let $q:(\Sigma,p)\to(Q,o)$ be an analytic envelope, and let
$\mathscr P_p$ be the isotropy at $p$ of a transverse pseudogroup
preserving the induced foliation. Descent defines a homomorphism
$\mathscr P_p\longrightarrow\Aut(Q,o)$.
Its kernel consists exactly of those transverse germs $h$ which fix every holomorphic first integral:
\begin{equation}\label{eq:pseudogroup-kernel}
h^*f=f
\text{ for every }f\in\cO_{\Sigma,p}^{\cF}.
\end{equation}
The effective holonomy on the envelope is the quotient of transverse holonomy by the subgroup acting trivially on all holomorphic first integrals.
\end{corollary}

\begin{proof}
Every first integral is a pullback from $Q$. Hence $q\circ h=q$
implies \eqref{eq:pseudogroup-kernel}. Conversely, if $h$ fixes every
first integral, then it fixes all components of $q$ in a local
embedding of $Q$, so $q\circ h=q$. Uniqueness of factorisation gives
$\bar h=\operatorname{id}_Q$.
\end{proof}

\begin{corollary}\label{cor:envelope-equivalence-invariant}
Let $h:(\Sigma_a,\cF_a,p_a)\to(\Sigma_b,\cF_b,p_b)$ be a
foliated biholomorphism between germs with analytic envelopes
$(Q_a,o_a)$ and $(Q_b,o_b)$.
If $\mathscr P_a$ is a transverse pseudogroup at $p_a$ and
$\mathscr P_b=h\mathscr P_a h^{-1}$, then the envelope isomorphism
$\bar h:Q_a\to Q_b$ conjugates the descended actions and their
effective quotients. Thus the envelope with its effective action is
invariant under equivalence of transverse germs equipped with pseudogroups.
\end{corollary}

\begin{proof}
Proposition~\ref{prop:pseudogroup-descent} gives the envelope isomorphism $\bar h$.  For $g\in\mathscr P_a$, functoriality of descent gives $\overline{hgh^{-1}}=\bar h\,\bar g\,\bar h^{-1}$, so $\bar h$ conjugates the descended actions.  Corollary~\ref{cor:effective-pseudogroup} identifies their kernels with the germs acting trivially on all holomorphic first integrals; the conjugacy passes to the effective quotients.
\end{proof}

This invariant agrees with Haefliger's in the regular case
\cite{Haefliger58,Haefliger84}. It applies to the intrinsic transverse
pseudogroup of \cite{ACSV}. For singular holonomy in the sense of
\cite{AZ}, assume holomorphic foliated representatives are available.
Their ambiguity consists of residual leafwise flows, which fix every
first integral and therefore have trivial descent. Two holomorphic
representatives have the same descent exactly when they agree on all
first integrals. For diagonal actions, transport is explicit.

Fix a coordinate support $I\ne\varnothing$ with $r_I<k$. Define
$$
V_I=\g/\h_I,\quad K_I=\{t:\Lambda_i(t)\in2\pi i\Z\ (i\in I)\},
\quad\Gamma_I=K_I/\h_I,\quad M_I=\operatorname{gp}(S_I).
$$
The injective map $V_I\to\C^I$ induced by the support weights
identifies $\Gamma_I$ with its intersection with $2\pi i\Z^I$.
It is discrete, and the orbit with its intrinsic manifold structure is
$L_I=V_I/\Gamma_I$. If the support is independent, then
$L_I=(\C^*)^I$ and $\Gamma_I=2\pi i\Z^I$. This intrinsic topology
is relevant when an orbit is not embedded.

For $\alpha\in M_I$, the covector
$\nu_\alpha=\sum_{j\notin I}\alpha_j\Lambda_j$ vanishes on $\h_I$,
so belongs to $V_I^*$. Write $T_I=\Hom(M_I,\C^*)$.

\begin{theorem}\label{thm:transport}
The ambient diagonal action induces transport on the toric quotient $Q_I$ by
$\tau_I:V_I\to T_I$, with
\begin{equation}\label{eq:transport}
\tau_I(v)^*X^\alpha=e^{\nu_\alpha(v)}X^\alpha,
\qquad \rho_I=\tau_I|_{\Gamma_I}.
\end{equation}
The effective monodromy group is $\Gamma_I/\ker\rho_I$ and preserves
the coordinate supports of the closed orbits in Proposition~\ref{prop:fibres}.
For $I\subset I'$, the projection $\pi_{I',I}:V_{I'}\to V_I$ sends
$\Gamma_{I'}$ to $\Gamma_I$ and makes $\vartheta_{I,I'}$ equivariant
for transport and monodromy. Periods act trivially on the factor
$\C^{|I|-r_I}$ foliated by points.
\end{theorem}
\begin{proof}
Let $v\in V_I$ and $\alpha\in M_I$. Because
$\nu_{\alpha}$ is additive, the map
$\alpha\mapsto e^{\nu_{\alpha}(v)}$ is a character of
$M_I$, and hence determines a point of $T_I$. Its action on
$\C[S_I]$ gives~\eqref{eq:transport}; it is holomorphic since a
finite Hilbert basis generates the coordinate algebra. Given a lift
$t\in\g$ of $v$, ambient diagonal transport is
$D_tw=(e^{\Lambda_j(t)}w_j)_{j\in J}$, and
\[
(D_tw)^\alpha
=\prod_{j\in J} e^{\alpha_j\Lambda_j(t)}w_j^{\alpha_j}
=e^{\nu_\alpha(v)}w^\alpha.
\]
The equality is independent of the lift because $\nu_\alpha$
vanishes on $\h_I$. Therefore $q_I D_t=\tau_I(v)q_I$.
Since $V_I$ is simply connected, a loop in $L_I$ lifts to a path
whose endpoint differs from its starting point by a period
$\gamma\in\Gamma_I$. The corresponding return map is
$\tau_I(\gamma)$. The action fixes every $X^\alpha$ exactly
when $e^{\nu_\alpha(\gamma)}=1$ for every $\alpha\in S_I$,
which identifies the effective isotropy. Multiplication by nonzero
scalars preserves supports, and it commutes with the residual torus
$G_I$. Thus it preserves the coordinate supports of the closed orbits.

For $I\subset I'$, let $D=I'\setminus I$ and
$\pi=\pi_{I',I}$. On coordinate algebras we check
\[
\tau_{I'}(v')^*\operatorname{res}_{I,I'}
=\operatorname{res}_{I,I'}\tau_I(\pi v')^*.
\]
If an exponent $\alpha\in S_I$ meets $D$, then both sides vanish.
Otherwise it lies in $S_{I'}$ and both actions multiply $X^\alpha$
by $e^{\nu_\alpha(v')}$: the same ambient covector
$\sum_{j\in J'}\alpha_j\Lambda_j$ descends to both quotient
spaces. The equality on monomials proves equivariance.
Finally $K_{I'}\subset K_I$, so the projection sends periods to
periods and the argument also proves monodromy equivariance. The
period acts trivially on the factor foliated by points because every
support coordinate multiplier is one. The identities compose for
three nested supports by functoriality of the quotient projections.
\end{proof}

\begin{proposition}\label{prop:arithmetic}
Put $M_I^{\mathrm{inv}}=\{\alpha\in M_I:
\nu_\alpha(\Gamma_I)\subset2\pi i\Z\}$ and
$D_I=\overline{\rho_I(\Gamma_I)}^{\mathrm{Zar}}\subset T_I$. Then
$$
X^*(D_I)=M_I/M_I^{\mathrm{inv}},\qquad
\dim D_I=\rk M_I-\rk M_I^{\mathrm{inv}}.
$$
Monodromy is trivial exactly when $M_I^{\mathrm{inv}}=M_I$, finite
exactly when this inclusion has finite index, and Zariski dense in
$T_I$ exactly when $M_I^{\mathrm{inv}}=0$. On a Hilbert basis
$\beta^a$, triviality and finiteness are respectively equivalent to
$\nu_{\beta^a}(\Gamma_I)\subset2\pi i\Z$ and
$\nu_{\beta^a}(\Gamma_I)\subset2\pi i\Q$ for every $a$.
\end{proposition}
\begin{proof}
The characters trivial on the monodromy image are precisely
$M_I^{\mathrm{inv}}$; character duality for diagonalisable groups
gives the formulas, including possible torsion in the quotient.
The last criterion uses the finite generation of both $M_I$ and
$\Gamma_I$, so rational periods have a common denominator.
\end{proof}

\subsection{Flat connection and residues}
Let $\omega_\alpha$ be the holomorphic form on $L_I$ whose pullback
to $V_I$ is $d\nu_\alpha$. We use the pullback representation on the
invariant algebra; the transport group is abelian, so these pullbacks
form a representation. On
$\mathscr A_I=\cO_{L_I}\otimes_\C\C[S_I]$ set
\begin{equation}\label{eq:connection}
\nabla_I(fX^\alpha)=(df-f\omega_\alpha)X^\alpha.
\end{equation}

\begin{theorem}\label{thm:connection}
The suspension
$$
\mathscr Q_I=(V_I\times Q_I)/\Gamma_I,\qquad
\gamma(v,q)=(v+\gamma,\rho_I(\gamma)^{-1}q),
$$
is an analytic bundle over $L_I$ with locally constant transitions.
It is holomorphically trivial by
$[v,q]\mapsto([v],\tau_I(v)q)$ and has geometric monodromy $\rho_I$.
Pullback transport on the invariant algebra is represented by the flat
connection \eqref{eq:connection}, with monodromy $\rho_I^*$.
All restriction homomorphisms are horizontal over
$\bar\pi_{I',I}:L_{I'}\to L_I$.

For independent supports, the coefficients $C_I$ of the introduction
give
\begin{equation}\label{eq:logconnection}
\nabla_I=d-\sum_{i\in I}\frac{dz_i}{z_i}\delta_i,
\qquad\delta_i(X^\alpha)=c_i(\alpha)X^\alpha.
\end{equation}
This extends logarithmically to the constant envelope algebra over
$(\mathbb P^1)^I$, with residues $-\delta_i$ at $z_i=0$ and
$\delta_i$ at $z_i=\infty$. Thus local monodromy is
$e^{2\pi i\delta_i}$.
\end{theorem}
\begin{proof}
The additive group $\Gamma_I$ is a discrete subgroup of the vector
space $V_I$, and its translation action on $V_I$ is properly
discontinuous. Its diagonal action on $V_I\times Q_I$ is therefore
properly discontinuous as well: any compact subset meets only
finitely many translates when projected to $V_I$. The quotient is
a complex analytic bundle over $L_I=V_I/\Gamma_I$, with the
holomorphic toric fibre $Q_I$ even when that fibre is singular.
The transition maps are locally constant in the chosen universal
cover, hence define a flat structure with monodromy $\rho_I$.

The map
\[
\mathscr Q_I\longrightarrow L_I\times Q_I,\qquad
[v,q]\longmapsto([v],\tau_I(v)q)
\]
is well-defined, because
$\tau_I(v+\gamma)\rho_I(\gamma)^{-1}=\tau_I(v)$.
Its inverse sends $([v],q)$ to $[v,\tau_I(-v)q]$ and is
independent of the chosen representative $v$. Both maps are holomorphic. This trivialises the analytic bundle;
a flat trivialisation exists exactly when monodromy is trivial.

For $\alpha\in S_I$, the additive homomorphism
$\nu_\alpha:V_I\to\C$ has differential invariant under
translations by $\Gamma_I$; it descends to a closed holomorphic
one-form $\omega_\alpha$ on $L_I$. Its additivity,
$\omega_{\alpha+\beta}=\omega_\alpha+\omega_\beta$, gives
\[
\nabla_I(fgX^{\alpha+\beta})
=gX^\beta\nabla_I(fX^\alpha)
+fX^\alpha\nabla_I(gX^\beta).
\]
The connection is flat since each $d\omega_\alpha=0$ and the coefficients commute.
A horizontal section $fX^\alpha$ satisfies $df=f\omega_\alpha$;
along a lifted path the coefficient $f$ is multiplied by
$e^{\nu_\alpha(v_1-v_0)}$.
On a loop representing $\gamma\in\Gamma_I$, this is $\rho_I(\gamma)^*$.

If the support weights are independent, then their dual coordinates on
$V_I$ identify the orbit with $(\C^*)^I$. In these coordinates
$\nu_\alpha=\sum_i c_i(\alpha)\Lambda_i$ and
$\omega_\alpha=\sum_i c_i(\alpha)dz_i/z_i$, proving
\eqref{eq:logconnection}. Each logarithmic form has residue $1$ at
zero and residue $-1$ at infinity on the respective $\mathbb P^1$.
With the sign of the connection, the residues are therefore
$-\delta_i$ at zero and $\delta_i$ at infinity. The horizontal
multiplier about zero is $e^{2\pi i\delta_i}$. Finally the support
restriction of a monomial either vanishes or preserves its ambient
covector $\nu_\alpha$, and hence its one-form; this establishes
horizontality of the restriction maps.
\end{proof}

Logarithmic residues are derivations of the envelope algebra;
monodromy records their exponentials. An integral shift of a residue
coefficient preserves its character multiplier. The finite-cover
criterion makes this relation explicit.

\begin{proposition}
\label{prop:cover-criterion}
For an independent support and Hilbert basis
$\beta^1,\ldots,\beta^N$, the following are equivalent:
\begin{enumerate}[(i)]
\item the envelope monodromy has finite image;
\item $c_i(\beta^a)\in\Q$ for all $i$ and $a$;
\item there are positive integers $N_i$ such that, on the base
cover $z_i=u_i^{N_i}$, the pulled-back connection is trivial as a
flat algebra connection.
\end{enumerate}
The cover extends to $(\mathbb P^1)^I$, ramified only at the boundary.
The gauge is algebraic on the covering torus.
\end{proposition}
\begin{proof}
The character $X^{\beta^a}$ is multiplied by
$\exp(2\pi i c_i(\beta^a))$ around the $i$th coordinate loop.
The finitely many generators of the semigroup algebra make the
representation finite precisely when each multiplier has finite
order, which is equivalent to the stated rationality condition.
Choose $N_i$ clearing all denominators and put
$z_i=u_i^{N_i}$. On the pullback, the connection is
\[
 d-\sum_{i\in I}N_i\,\delta_i\,\frac{du_i}{u_i}.
\]
The integers $N_i c_i(\alpha)$ are additive in $\alpha$, so the
gauge $X^\alpha\mapsto u^{-NC_I(\alpha)}X^\alpha$ carries the
connection to $d$. Conversely, a
trivialising finite cover annihilates a subgroup of finite index in
$\pi_1((\C^*)^I)$, so the original monodromy image is finite.
\end{proof}

\begin{proposition}\label{prop:deRham}
For independent $I$, set $S_I^{\mathrm{inv}}=S_I\cap M_I^{\mathrm{inv}}$
and $A_I=\C[z_i^{\pm1}:i\in I]\otimes\C[S_I]$. Then
$$
H^j_{\mathrm{dR}}((\C^*)^I,A_I,\nabla_I)
=\C[S_I^{\mathrm{inv}}]\otimes\bigwedge^j\C^I.
$$
For $\alpha\in S_I^{\mathrm{inv}}$, the class of
$X^\alpha\otimes(e_{i_1}\wedge\cdots\wedge e_{i_j})$
is represented by $z^{C_I(\alpha)}X^\alpha
(dz_{i_1}/z_{i_1})\wedge\cdots\wedge(dz_{i_j}/z_{i_j})$.
\end{proposition}
\begin{proof}
Write $z^m=\prod_{i\in I}z_i^{m_i}$ for
$m\in\Z^I$. Since the connection preserves the character
decomposition of $\C[S_I]$, its algebraic de Rham complex decomposes
as the direct sum over pairs $(m,\alpha)$ of complexes
\[
\C\,z^mX^\alpha\otimes\bigwedge^\bullet\C^I.
\]
In the basis $dz_i/z_i$, the differential on this summand is
exterior multiplication by
$\sum_{i\in I}(m_i-c_i(\alpha))\,dz_i/z_i$. If this covector is
nonzero, choose an index $i$ on which it has a nonzero coefficient
and contract with the corresponding dual vector, divided by that
coefficient. Cartan's elementary identity shows that this is a
contracting homotopy. Thus the summand is acyclic.

The differential vanishes precisely when
$m=C_I(\alpha)\in\Z^I$. These pairs are indexed by the
exponents $\alpha\in S_I^{\mathrm{inv}}$ of monodromy-invariant monomials.
Their surviving representatives are
\[
z^{C_I(\alpha)}X^\alpha\,
\frac{dz_{i_1}}{z_{i_1}}\wedge\cdots\wedge
\frac{dz_{i_j}}{z_{i_j}}.
\]
The multiplication of cohomology classes corresponds to addition of
invariant exponents and exterior multiplication of differentials,
which gives the asserted graded algebra isomorphism.
\end{proof}

\section{Semiglobal holomorphic integrability}\label{sec:integrability}
Fix a coordinate support $I\ne\varnothing$ with $r_I=|I|<k$, and write
$J=I^c$, $q=n-k$. Its leaf is
$L_I=(\C^*)^I\times\{0\}\subset U_I=(\C^*)^I\times\C^J$.
We use germs \emph{along the entire leaf}. With restriction maps
and $U$ ranging over open neighbourhoods of $L_I$ in
$(\C^*)^I\times\C^J$, define
\begin{equation}\label{eq:BI-definition}
\mathcal B_I=\varinjlim_{U\supset L_I}
\bigl\{f\in\mathcal O(U):X_v(f)=0\ \text{for all }v\in\g\bigr\}.
\end{equation}

The transverse radius of a representative may vary along $L_I$.
Two representatives define the same germ if they agree near the whole
leaf. Complete integrability means that $q$ such first integrals have
generically independent differentials.
Put
$$
S_I^{\mathrm{inv}}=\{\alpha\in S_I:C_I(\alpha)\in\Z^I\},\quad
M_I^{\mathrm{sg}}=\operatorname{gp}(S_I^{\mathrm{inv}}),\quad
Q_I^{\mathrm{sg}}=\Spec\C[S_I^{\mathrm{inv}}].
$$
In general $M_I^{\mathrm{sg}}$ is smaller than $M_I^{\mathrm{inv}}$.

\begin{theorem}\label{thm:semiglobal}
The semigroup $S_I^{\mathrm{inv}}$ is finitely generated and saturated
in its group completion, and
\begin{equation}\label{eq:semiglobal}
\mathcal B_I\simeq\cO_{Q_I^{\mathrm{sg}},o},\qquad
X^\alpha\longmapsto F_\alpha(z,w)=z^{-C_I(\alpha)}w^\alpha.
\end{equation}
The functions $F_\beta$ for a Hilbert basis define a map
$\Phi_I:U_I\to Q_I^{\mathrm{sg}}$, sending $L_I$ to $o$, universal
for leaf-constant holomorphic map germs along $L_I$. The maximum
functional rank of $\mathcal B_I$ is
$r_{\mathrm{sg}}=\rk M_I^{\mathrm{sg}}$.
\end{theorem}
\begin{proof}
Let $f$ be a holomorphic first integral defined on a neighbourhood
of the whole leaf $L_I$. Along $w=0$ the transverse Taylor
coefficients
\[
a_\alpha(z)=\frac{1}{\alpha!}
\left.\frac{\partial^{|\alpha|}f}{\partial w^\alpha}
\right|_{w=0}
\]
are single-valued holomorphic functions on
$(\C^*)^I$. Local Taylor expansion gives
$f(z,w)=\sum_\alpha a_\alpha(z)w^\alpha$ in a transverse
polydisc depending on the base point. The equations of the
$\h_I$-action force $a_\alpha=0$ unless
$\alpha\in S_I$. For each such exponent, the remaining
infinitesimal equations become
\[
z_i\frac{\partial a_\alpha}{\partial z_i}
=-c_i(\alpha)a_\alpha,\qquad i\in I.
\]
On the universal cover $z_i=e^{t_i}$ these equations have solutions
$a_\alpha(e^t)=b_\alpha e^{-\sum_i c_i(\alpha)t_i}$.
Such a solution descends to the base precisely when
$e^{-2\pi i c_i(\alpha)}=1$ for every $i$ or $b_\alpha=0$.
Every nonzero Taylor coefficient is indexed by $S_I^{\mathrm{inv}}$
and is the stated monomial
$b_\alpha z^{-C_I(\alpha)}$.

The group $M_I^{\mathrm{inv}}$ is the kernel of a character
representation of $M_I$, hence a subgroup of a finitely generated
free abelian group. The cone
$(M_I^{\mathrm{inv}}\otimes\R)\cap\R_{\ge0}^J$ is rational
polyhedral relative to this lattice. Gordan's lemma gives finite
generation of $S_I^{\mathrm{inv}}$; the coordinatewise argument
used for $S_I$ gives saturation in its group completion.
If $S_I^{\mathrm{inv}}=\{0\}$, then the coefficient equations above
show that every semiglobal first integral is constant; the quotient
$Q_I^{\mathrm{sg}}$ is a point and the universal property and rank
assertions follow. For the convergence argument, we may therefore
assume $S_I^{\mathrm{inv}}\ne\{0\}$. Let
$\beta^1,\ldots,\beta^N$ be its Hilbert basis. The functions
$F_{\beta^a}=z^{-C_I(\beta^a)}w^{\beta^a}$ satisfy the
toric binomial relations, so they define an analytic map
$\Phi_I$ on $U_I$, taking the entire leaf to the distinguished
point of $Q_I^{\mathrm{sg}}$.

At the base point $z=(1,\ldots,1)$, the series
$\sum_{\alpha\in S_I^{\mathrm{inv}}}b_\alpha w^\alpha$
is convergent, since it is the Taylor series of the restriction
of the original representative to that transversal. The Hilbert-basis
convergence argument in
Theorem~\ref{thm:envelope} therefore produces a unique
$g\in\cO_{Q_I^{\mathrm{sg}},o}$ with those coefficients.
To establish extension along the non-compact leaf, choose a
representative of $g$ on an open neighbourhood $V$ of $o$ in
$Q_I^{\mathrm{sg}}$. Since the Hilbert-basis components of
$\Phi_I$ are holomorphic on $U_I$ and vanish along $L_I$,
the set $\Phi_I^{-1}(V)$ is an open neighbourhood of the
\emph{entire} leaf. The composite $g\circ\Phi_I$ is holomorphic
there and is annihilated by the ambient vector fields. This
argument imposes no uniform transverse radius: the neighbourhood
$\Phi_I^{-1}(V)$ may become arbitrarily thin towards the ends of
$L_I$.
The Taylor coefficients of $g\circ\Phi_I$ and $f$ agree in the
chosen transversal and solve the same linear differential equations,
so the two germs along $L_I$ coincide. Conversely every such
composite is a first integral. Injectivity follows by restricting to
$z=1$, while factorisation of target coordinate functions and their
defining equations gives the universal mapping property.

Finally, a choice of $r=\rk M_I^{\mathrm{sg}}$ independent
exponents $\alpha^1,\ldots,\alpha^r\in S_I^{\mathrm{inv}}$
yields independent logarithmic differentials of the monomials
$z^{-C_I(\alpha^a)}w^{\alpha^a}$ on the open coordinate torus:
their exponent vectors $(-C_I(\alpha^a),\alpha^a)$ are
independent because their projections to the $w$-coordinates are.
No collection of more than $r$ first integrals can be functionally
independent since they all factor through the $r$-dimensional normal
toric germ $Q_I^{\mathrm{sg}}$.
\end{proof}

Scalar first integrals contain the inverse multiplier
$z^{-C_I(\alpha)}$; horizontal sections of the algebra connection in
Proposition~\ref{prop:deRham} contain $z^{C_I(\alpha)}$.
The former use inverse pullback of fibrewise functions; both conventions
give the same single-valuedness condition.

\begin{corollary}\label{cor:complete}
There are $q$ independent semiglobal holomorphic first integrals if
and only if $\dim Q_I=q$ and $\rho_I(\Gamma_I)$ is finite. More
generally $Q_I^{\mathrm{sg}}=Q_I\!\!/D_I$, where $D_I$ is the
Zariski closure of monodromy. For finite image $F$, this is $Q_I/F$.
\end{corollary}
\begin{proof}
The inclusions
$M_I^{\mathrm{sg}}\subset M_I^{\mathrm{inv}}\subset M_I$
and $\rk M_I\le q$ show that $r_{\mathrm{sg}}=q$ forces $\rk M_I=q$ and finite index of $M_I^{\mathrm{inv}}$.
Proposition~\ref{prop:arithmetic} gives finite monodromy. Conversely,
if its image has exponent $N$, then $NS_I\subset S_I^{\mathrm{inv}}$,
so both semigroups have the same rank. The quotient assertion follows
by comparing invariant monomials.
\end{proof}

\begin{proposition}
\label{prop:holonomy-invariant-germs}
The restriction of semiglobal integrals to a fixed transversal
identifies $\mathcal B_I$ with the invariant local algebra
\[
\mathcal O_{Q_I,o_I}^{\rho_I(\Gamma_I)}
=\mathcal O_{Q_I,o_I}^{D_I}
\simeq\mathcal O_{Q_I^{\mathrm{sg}},o}.
\]
Here $Q_I^{\mathrm{sg}}=Q_I\!\!/D_I$ is the affine invariant-theoretic quotient.
\end{proposition}
\begin{proof}
Each element of $\mathcal O_{Q_I,o_I}$ has a convergent
character expansion $\sum_{\alpha\in S_I}b_\alpha X^\alpha$
by Theorem~\ref{thm:envelope}. The diagonal monodromy action
multiplies $X^\alpha$ by
$\exp(\nu_\alpha(\gamma))$. Uniqueness of the expansion
shows that invariance under every period is equivalent to
$\nu_\alpha(\Gamma_I)\subset2\pi i\Z$ for each
nonzero coefficient. These are precisely the exponents in
$S_I^{\mathrm{inv}}$. The same character equations define
invariance under the Zariski closure $D_I$.
Finally the convergent Hilbert-basis factorisation for
$S_I^{\mathrm{inv}}$ identifies this invariant algebra with the
local analytic algebra of $Q_I^{\mathrm{sg}}$. The restriction
of a semiglobal integral to the chosen transversal gives this
identification by Theorem~\ref{thm:semiglobal}.
\end{proof}

Two elementary configurations illustrate the difference between
finite monodromy and infinite monodromy with surviving integrals.
First take three weights in $(\C^2)^*$
\[
\Lambda_z=(1,0),\qquad
\Lambda_x=(a,1),\qquad
\Lambda_y=(b,-1),\qquad
 a+b=\frac p\ell>0,
\]
where $p,\ell$ are coprime positive integers and $a,b>0$.
The configuration is in linear general position and lies in the
Poincar\'e domain. Along $I=\{z\}$ the residual weights
are $(1,-1)$, so $S_I=\N(1,1)$ and $Q_I=\C$ with
coordinate $u=xy$. The monodromy multiplies $u$ by
$e^{2\pi i p/\ell}$; it has order $\ell$. The semigroup of invariant exponents is
$\ell\N(1,1)$, and the first-integral algebra along the whole
leaf is
\[
\mathcal B_I=\C\bigl\{z^{-p}(xy)^\ell\bigr\}.
\]
For comparison, take weights
$\Lambda_z=(1,0)$,
$\Lambda_x=(\gamma/2,1)$,
$\Lambda_y=(\gamma/2,-1)$ and
$\Lambda_w=(1,0)$, where
$\gamma=1+\log 2/(2\pi i)$.
At the support $I=\{z\}$ the local envelope is $\C^2$ with
coordinates $(u,v)=(xy,w)$, and monodromy acts by
$(u,v)\mapsto(2u,v)$. Hence
$S_I^{\mathrm{inv}}=\N(0,0,1)$ and
$\mathcal B_I=\C\{w/z\}$: precisely one of the two
local integrals survives.

The exact loss of integrals is
\begin{equation}\label{eq:defect}
q-r_{\mathrm{sg}}=(q-\rk M_I)
+(\rk M_I-\rk M_I^{\mathrm{inv}})
+(\rk M_I^{\mathrm{inv}}-\rk M_I^{\mathrm{sg}}).
\end{equation}
All terms are nonnegative. The middle term is $\dim D_I$; the last
measures the additional restriction imposed by holomorphic extension
at $o$. The affine quotient $Q_I\!\!/D_I$ can therefore have smaller
dimension than the field of rational invariants would suggest:
$\dim(Q_I\!\!/D_I)$ need not equal
$\dim Q_I-\dim D_I$.

For example, take coordinates $(z,x,y,w)$ and weights
$$
\Lambda_z=(1,0),\quad \Lambda_x=(\gamma/2,1),\quad
\Lambda_y=(\gamma/2,-1),\quad\Lambda_w=(\gamma,0),\qquad
\gamma=1+\frac{\log2}{2\pi i}.
$$
They lie in the Poincar\'e domain. At $I=\{z\}$ the residual weights
are $(1,-1,0)$ and $Q_I=\C^2$ with $(u,v)=(xy,w)$.
Monodromy is $(u,v)\mapsto(2u,2v)$, so
$M_I^{\mathrm{inv}}=\Z(1,1,-1)$ but $S_I^{\mathrm{inv}}=\{0\}$.
There is no nonconstant semiglobal holomorphic integral, although
$xy/w$ is meromorphic and invariant. The three terms in~\eqref{eq:defect} are $0,1,1$.

\Needspace{9\baselineskip}
\section{Realisation and explicit envelopes}\label{sec:realisation}
The residual action and envelope monodromy can be chosen independently,
even under the Poincar\'e-domain condition and linear general position.

\begin{lemma}\label{lem:lifts}
Let $d,r\ge1$ and let $\beta_1,\ldots,\beta_s\in(\C^d)^*$
be in linear general position. For arbitrary $a_j^0\in(\C^r)^*$ there are integer
covectors $n_j$ such that the weights
$$
(\varepsilon_i,0)\ (1\le i\le r),\qquad
(a_j^0+n_j,\beta_j)\ (1\le j\le s)
$$
are in linear general position in dimension $r+d$ and lie in the
Poincar\'e domain, separated by $(1,\ldots,1,0)$.
\end{lemma}
\begin{proof}
Write $E^*=\C^r$ and $F^*=\C^d$. The support weights are the
standard basis of $E^*$, while the remaining weights are
$(a_j,\beta_j)$. Linear general position requires the independence of
all subsets of at most $r+d$ weights. Fix $A\subset\{1,\ldots,r\}$
and a subset $K$ of residual indices with $|A|+|K|\le r+d$.
Modulo the span of the support vectors indexed by $A$, this becomes
an independence condition in
\[
(E^*/\operatorname{Span}\{\varepsilon_i:i\in A\})\oplus F^*.
\]
When $|K|\le d$, the $F^*$-components are already independent.
When $|K|=d+t>d$, the inequality on the number of weights gives
$t\le r-|A|$. Choose $d$ indices $K_0\subset K$; their residual
weights form a basis of $F^*$. For the remaining $t$ indices choose
their support components, modulo the chosen support span, to be
independent, and set the projected support components of the
$K_0$-weights to zero. Subtracting appropriate combinations of the
$K_0$ vectors from the remaining lifted weights produces $d$
independent vectors in $F^*$ and $t$ independent vectors in the
support quotient. Thus the rank condition holds for some choice of
the $a_j$ and is the non-vanishing of a nonzero minor polynomial.

There are finitely many pairs $(A,K)$. The product of one selected
nonzero minor for each is therefore a nonzero polynomial $P$ in
the support coefficients. Since a product of infinite arithmetic
progressions is Zariski dense in affine space, $P$ cannot vanish
at all points of a sufficiently positive translated integral grid.
Consequently we can choose integer perturbations $n_j$ with
$P(a^0+n)\ne0$ and all the real parts of
$(a_j^0+n_j)(1,\ldots,1)$ positive. The first condition is linear general position and the second says that $(1,\ldots,1,0)$ separates
the weights from the origin.
\end{proof}

\begin{theorem}
\label{thm:realisation}
Let $d,r\ge1$, and let $\beta_1,\ldots,\beta_s$ span $(\C^d)^*$.
They occur along an $r$-dimensional singular orbit of an ambient configuration
in linear general position in the Poincar\'e domain if and only if they
are in linear general position in dimension $d$. In this case, with
$S_\beta$ its resonance semigroup and
$T_\beta=\Hom(\operatorname{gp}(S_\beta),\C^*)$, every homomorphism
$\eta:\Z^r\to T_\beta$ can be realised as envelope monodromy while
keeping the residual weights fixed.
\end{theorem}
\begin{proof}
If an ambient configuration is in linear general position, then adjoining
all its $r$ independent support weights to any set of at most $d$
residual weights gives an independent family. Quotienting by the
support span preserves the independence of their residual classes,
which proves necessity.

Assume conversely that the $\beta_j$ are in linear general position.
Put $M_\beta=\operatorname{gp}(S_\beta)$. Given
$\eta:\Z^r\to T_\beta$, write
$\eta_i=\eta(e_i)\in\Hom(M_\beta,\C^*)$ for its values on the
standard basis of the period lattice. Fix an integral basis
$m_1,\ldots,m_\ell$ of $M_\beta$ and, for every support index
$i$, choose logarithms $\ell_i(m_a)\in\C$ satisfying
$e^{2\pi i\ell_i(m_a)}=\eta_i(m_a)$. Extend the assignment by
$\Z$-linearity to $M_\beta$. The elements of an integral basis of $M_\beta\subset\Z^s$
are complex-linearly independent: a nonzero minor of their integer
coordinate matrix remains nonzero over $\C$. Thus each assignment
extends to a complex-linear functional on $\C^s$. Write the values
of these extensions on the standard source basis as
$a_j^0=(a^0_{1j},\ldots,a^0_{rj})$.

In $(\C^{r+d})^*=\C^r\oplus(\C^d)^*$ consider the weights
\[
\Lambda_i=(e_i,0)\quad(1\le i\le r),\qquad
\Lambda_{r+j}=(a_j^0,\beta_j)\quad(1\le j\le s).
\]
At the support $I=\{1,\ldots,r\}$, the residual weights are
exactly $\beta_j$, and for every $\alpha\in M_\beta$ the
coefficients of the ambient relation along the support are
$C_I(\alpha)=(\ell_i(\alpha))_{i=1}^r$. Hence the $i$th
period acts on the character $X^\alpha$ by
$e^{2\pi i\ell_i(\alpha)}=\eta_i(\alpha)$.

These weights need not yet be in linear general position or lie in the
Poincar\'e domain. Apply Lemma~\ref{lem:lifts}: replace the support
components by $a_j^0+n_j$ with integer covectors $n_j$ chosen so
that both properties hold. This leaves the residual weights
unchanged. Its effect on the coefficient of any integral relation
$\alpha$ is the integer
$\sum_j\alpha_j n_j$. Exponentiating through $2\pi i$
therefore leaves all prescribed character multipliers unchanged.
Since the characters indexed by $M_\beta$ separate the points of
$T_\beta$, the monodromy representation of the constructed action
is $\eta$.
\end{proof}

Keep a residual configuration $\beta$ in linear general position.
Fix $r\ge1$ and set
\begin{equation}\label{eq:monodromy-parameter-space}
\mathcal R_{\beta,r}=\Hom(\Z^r,T_\beta)\simeq T_\beta^r,
\end{equation}
where the isomorphism is evaluation on the standard basis of $\Z^r$.  For $\eta\in\mathcal R_{\beta,r}$, set
\begin{equation*}
M_\eta=\{m\in M_\beta:\eta(n)(m)=1\text{ for every }n\in\Z^r\},
\end{equation*}
and let $D_\eta$ be the Zariski closure of $\eta(\Z^r)$ in $T_\beta$.  For $m\in M_\beta$, let $\chi_m:T_\beta\to\C^*$ be evaluation at $m$, and let $T_\beta^{\mathrm{tors}}$ denote the torsion subgroup of $T_\beta$.
\begin{proposition}\label{thm:monodromy-parameter-space}
For every $\eta\in\mathcal R_{\beta,r}$,
\begin{equation}\label{eq:Deta-character-group}
X^*(D_\eta)\simeq M_\beta/M_\eta,
\end{equation}
and
\begin{equation}\label{eq:Deta-dimension}
\dim D_\eta=\rk M_\beta-\rk M_\eta.
\end{equation}
The representation $\eta$ is trivial if and only if $M_\eta=M_\beta$, has finite image if and only if $M_\eta$ has finite index in $M_\beta$, and is Zariski dense in $T_\beta$ if and only if $M_\eta=0$.  The Zariski-dense locus is
\begin{equation}\label{eq:dense-monodromy-locus}
\mathcal R_{\beta,r}^{\mathrm{dense}}
=\mathcal R_{\beta,r}\setminus
\bigcup_{0\ne m\in M_\beta}(\ker\chi_m)^r,
\end{equation}
whereas the finite-image locus is
\begin{equation}\label{eq:finite-monodromy-locus}
\mathcal R_{\beta,r}^{\mathrm{fin}}
=(T_\beta^{\mathrm{tors}})^r.
\end{equation}
Every point of $\mathcal R_{\beta,r}$ is realised as the envelope monodromy of a lift whose ambient weight configuration is in linear general position and lies in the Poincar\'e domain, with residual envelope $Q_\beta$.
\end{proposition}

\begin{proof}
Proposition~\ref{prop:arithmetic} gives
\eqref{eq:Deta-character-group}, \eqref{eq:Deta-dimension} and the
three image criteria. Under \eqref{eq:monodromy-parameter-space},
write $\eta=(t_1,\ldots,t_r)$. A nonzero $m$ belongs to $M_\eta$
exactly when every $t_i$ lies in $\ker\chi_m$, proving
\eqref{eq:dense-monodromy-locus}. The image is finite exactly when
all its generators $t_i$ have finite order, giving
\eqref{eq:finite-monodromy-locus}. Realisation follows from
Theorem~\ref{thm:realisation}.
\end{proof}

The two monodromy loci have the following properties.
\begin{corollary}\label{cor:monodromy-loci}
If $\rk M_\beta>0$, then the Zariski-dense locus in~\eqref{eq:dense-monodromy-locus} is the complement of a countable union of proper algebraic subsets of $\mathcal R_{\beta,r}$.  The finite-image locus in~\eqref{eq:finite-monodromy-locus} is Zariski dense in $\mathcal R_{\beta,r}$.
\end{corollary}

\begin{proof}
The lattice $M_\beta$ is countable.  For $0\ne m\in M_\beta$, the character $\chi_m$ is non-trivial, so $\ker\chi_m$ is a proper algebraic subgroup of $T_\beta$.  Formula~\eqref{eq:dense-monodromy-locus} gives the first assertion.
Choose an isomorphism $T_\beta\simeq(\C^*)^\ell$, where $\ell=\rk M_\beta$.  Let $\mu_\infty\subset\C^*$ be the group of all roots of unity.  The torsion subgroup of $T_\beta$ is $(\mu_\infty)^\ell$.  To prove that it is Zariski dense, let a Laurent polynomial vanish on $(\mu_\infty)^\ell$.  After multiplication by a monomial, assume that it is a polynomial.  For $\ell=1$, a nonzero polynomial cannot vanish at all roots of unity.  For $\ell>1$, fix torsion values in the first $\ell-1$ variables.  The resulting polynomial in the last variable vanishes at every root of unity, so each coefficient vanishes on $(\mu_\infty)^{\ell-1}$; induction shows that all coefficients are zero.  Thus $T_\beta^{\mathrm{tors}}$ is Zariski dense in $T_\beta$, and its $r$-fold product is Zariski dense in $T_\beta^r$.
\end{proof}

\subsection{Low-dimensional and determinantal examples}
For integer weights $a,-b$, with $a,b>0$ coprime, the envelope is
$q(x,y)=x^by^a:\C^2\to\C$. Its zero fibre contains both punctured
coordinate leaves and the origin. For weights $1,1,-2$, the Hilbert
generators $u=x_1^2z$, $v=x_1x_2z$, $w=x_2^2z$ give
$Q=\{v^2=uw\}$. An explicit Poincar\'e-domain lift in linear general position is $(1,0),(1,1),(2,1),(1,-2)$.

\begin{proposition}
\label{prop:rank-two-determinants}
For a diagonal $\C^2$-action with weights
$\Lambda_i=(a_i,b_i)$ and
$\Delta_{ij}=a_ib_j-b_ia_j$, assume $\Lambda_i\ne0$.
The residual weights at the support $\{i\}$ are, up to a common nonzero multiple,
$-\Delta_{ij}$ for $j\ne i$. Thus its resonance semigroup is
\[
S_i=\{\alpha\in\N^{\{1,\ldots,n\}\setminus\{i\}}:
\sum_{j\ne i}\alpha_j\Delta_{ij}=0\}.
\]
Conversely, every finite list of nonzero complex numbers occurs as
the residual weights at a one-dimensional singular coordinate stratum of a
Poincar\'e-domain diagonal $\C^2$-action in linear general position.
\end{proposition}
\begin{proof}
A generator of $\ker\Lambda_i$ is $(b_i,-a_i)$. Evaluating the
$j$th weight on this vector gives
$\Lambda_j(b_i,-a_i)=a_jb_i-b_ja_i=-\Delta_{ij}$.
The formula for $S_i$ follows by the monomial resonance condition.
In residual dimension one, linear general position is precisely the
non-vanishing of every residual weight; apply the realisation theorem
with $r=d=1$ to any such list. Alternatively, taking the first
support weight $(1,0)$, the residual weights $(\beta_j)$ are obtained
from ambient weights $(c_j,\beta_j)$; generic positive choices of
the $c_j$ avoid all vanishing pairwise determinants and place the
configuration in the Poincar\'e domain.
\end{proof}

\begin{proposition}\label{prop:segre}
For residual weights $+1$ on $x_1,\ldots,x_p$ and $-1$ on
$y_1,\ldots,y_t$, the envelope is the rank-one determinantal cone
$$
Q_{p,t}=\{Z\in\operatorname{Mat}_{p\times t}(\C):\rk Z\le1\},
\qquad z_{ab}=x_ay_b.
$$
It has dimension $p+t-1$ and is singular at its vertex if $p,t\ge2$.
Every $Q_{p,t}$ occurs along a singular orbit of an ambient
Poincar\'e-domain configuration in linear general position.
Consequently the envelope dimension can jump from zero at the origin
to any prescribed $N\ge1$ on a one-dimensional singular stratum;
for $N\ge3$ the latter envelope can be chosen singular.
\end{proposition}
\begin{proof}
An invariant monomial has equal total $x$- and $y$-degrees, so pairing
factors expresses it in the $x_ay_b$. The image of $(x,y)\mapsto xy^{\mathsf T}$
is the rank-one cone, whose ideal consists of the $2\times2$ minors.
Its dimension is $p+t-1$, whereas its tangent space at zero has
dimension $pt$. Realisation follows from Theorem~\ref{thm:realisation}.
Take $(p,t)=(1,N)$ for an arbitrary jump, and $(2,N-1)$ for the
singular examples. The origin envelope is a point by Poincar\'e
separation.
\end{proof}

Explicitly, use $\Lambda_0=(1,0)$,
$\Lambda_a=(a_a,1)$ and $\Lambda_{p+b}=(c_b,-1)$, with positive real
$a_a$ distinct among themselves and likewise for $c_b$.
Pairwise determinants are nonzero. The envelope monodromy is
$$z_{ab}\longmapsto e^{2\pi i(a_a+c_b)}z_{ab};$$
it is trivial precisely when all sums are integral, and finite precisely
when all are rational. The same formulas hold for complex $a_a,c_b$
with positive real parts and the corresponding distinctness conditions.

\begin{proposition}
\label{prop:rank-one-signs}
Suppose all residual weights are nonzero integers
$b_1,\ldots,b_s$, and let
$L=\{a\in\Z^s:\sum_j a_jb_j=0\}$ and
$S=L\cap\N^s$. If weights of both signs occur, then
$\operatorname{gp}(S)=L$ and $\dim Q=s-1$; if all weights have one
sign, then $S=\{0\}$ and $Q$ is a point.
\end{proposition}
\begin{proof}
For weights of a single sign, no nonnegative nonzero relation is
possible. In the mixed-sign case set
$A=-\sum_{b_j<0}b_j>0$ and $B=\sum_{b_j>0}b_j>0$, and define
$\eta_j=A$ on positive-weight indices and $\eta_j=B$ on
negative-weight indices. Then $\eta\in\N_{>0}^s$ and
$\sum_j b_j\eta_j=AB-BA=0$. For any $a\in L$, choose a positive
integer $N$ such that $a+N\eta$ is coordinatewise nonnegative.
Both $a+N\eta$ and $N\eta$ belong to $S$, and
$a=(a+N\eta)-N\eta$ belongs to $\operatorname{gp}(S)$.
Therefore $M=L$, a rank-$s-1$ lattice, giving the dimension
assertion.
\end{proof}

\section{Reconstruction of weight configurations}\label{sec:inverse}
We retain the quotient map together with its ordered source coordinates.
This labelled quotient map specifies
$S=\{\alpha:w^\alpha\in q^*\cO_{Q,o}\}\subset\N^s$.
Two labelled spanning weight configurations are linearly equivalent if
a linear isomorphism carries each weight to its counterpart.

\begin{theorem}\label{thm:ambiguity}
Fix $d\ge1$ and a semigroup $S\subset\N^s$ arising from a spanning
$d$-dimensional configuration. Set $M=\operatorname{gp}(S)$,
$R=M\otimes\C$, and $\delta=s-d-\rk M$.
The configurations with semigroup exactly $S$, modulo labelled linear
equivalence, are parametrised by
$$
\operatorname{Gr}_{\delta}(\C^s/R)
\setminus\bigcup_{\alpha\in\N^s\setminus S}
\{P:\bar\alpha\in P\}.
$$
This is a dense $G_\delta$ subset of a Grassmannian of dimension
$\delta d$. It is a single point exactly when $\delta=0$ and is
uncountable otherwise. Requiring linear general position excludes finitely many
additional incidence subvarieties; if this locus is nonempty, then it is
again dense $G_\delta$. Every configuration in that locus admits the
Poincar\'e realisations of Theorem~\ref{thm:realisation}.
\end{theorem}
\begin{proof}
Let $\mu:\C^s\to V$ be one spanning labelled configuration of
dimension $d$ with semigroup $S$, and put $K_0=\ker\mu$. Since
$S\subset K_0\cap\N^s$ and $R\subset K_0$, the latter kernel has
dimension $s-d$ and determines a $\delta$-plane $K_0/R$ in
$E=\C^s/R$, where $\dim E=d+\delta$.

We first verify that $R\cap\N^s=S$. Put
$L=K_0\cap\Z^s$, which is saturated. Let
$C=L_{\R}\cap\R_{\ge0}^s$. Its lattice points are $S$. The
complex span $R$ of $S$ is the complexification of the real span
of the cone $C$. Choosing a lattice point $\eta$ in its relative
interior, any lattice point $a$ in that real span has both
$a+N\eta$ and $N\eta$ in $C\cap L$ for all sufficiently large
$N$. Thus
\[
\operatorname{gp}(S)=L\cap\operatorname{Span}_{\R}(C),
\]
and the equality with $R\cap\Z^s$ follows. Intersecting with $\N^s$ gives $R\cap\N^s=S$.

Every $d$-dimensional spanning labelled configuration is obtained,
up to linear equivalence, from the quotient $\C^s/K$, where
$\dim K=s-d$. Requiring all the prescribed resonances is exactly
$R\subset K$, hence $K/R\in\operatorname{Gr}_{\delta}(E)$. It
has no additional positive integral relations precisely when
$K\cap\N^s=S$. For each $\alpha\in\N^s\setminus S$, the
image $\bar\alpha\in E$ is nonzero and the unwanted relation
occurs on the incidence locus
\[
\Sigma_\alpha=
\{P\in\operatorname{Gr}_{\delta}(E):\bar\alpha\in P\}.
\]
For $\delta>0$, this is a proper closed Schubert variety, in fact
isomorphic to $\operatorname{Gr}_{\delta-1}(E/\C\bar\alpha)$;
for $\delta=0$ it is empty. There are countably many such
conditions. The Grassmannian is compact and metrisable, hence a
Baire space, and every proper complex subvariety has empty interior.
The complement of their union is therefore dense $G_\delta$.
Its ambient dimension is
$\delta((d+\delta)-\delta)=\delta d$. When $\delta>0$ this
is positive, and a dense $G_\delta$ subset is uncountable.

Linear general position requires
$K\cap E_A=0$ for each coordinate subspace
$E_A=\operatorname{Span}\{e_j:j\in A\}$ with $1\le |A|\le d$.
If this locus is nonempty, then $R\cap E_A=0$, so the
image of each $E_A$ in $E=\C^s/R$ has dimension $|A|$.
The condition is therefore $P\cap\overline E_A=0$ in $E$.
These are finitely many further Schubert incidence conditions in
the quotient $E$, and are proper whenever the linear general position
locus is nonempty. Their removal preserves the preceding Baire
conclusion. The realisation result follows from
Theorem~\ref{thm:realisation}.
\end{proof}

At support $I$, the residual weights span $\h_I^*$, so
$$
\dim Q_I\le\dim\ker\mu_I
=n-k-(|I|-r_I)\le n-k,\qquad
\mu_I:\C^J\to\h_I^*,\quad e_j\mapsto\mu_j.
$$
Equality $\dim Q_I=n-k$ forces
$r_I=|I|$ and $R_I=M_I\otimes\C=\ker\mu_I$.
Thus the labelled classes of $e_j$ in $W_I=\C^J/R_I$ recover the
residual weights. For $|I|=r_I=k-1$, the parameter space is
$\mathbb P^{\delta_I}$, where $\delta_I=n-k-\dim Q_I$; linear
equivalence here is common nonzero scaling.

\begin{theorem}\label{thm:reconstruction}
At a singular coordinate stratum indexed by $I\ne\varnothing$ with
$\dim Q_I=n-k$, the quotient
map with labelled residual coordinates and the logarithmic connection
with labelled support coordinates determine all ambient weights up to
linear equivalence. Explicitly, extend
$C_I^\C:R_I\to\C^I$ to any linear map $A:\C^J\to\C^I$ and put
\begin{equation}\label{eq:reconstruct}
\widehat\Lambda_i=(e_i,0)\quad(i\in I),\qquad
\widehat\Lambda_j=(Ae_j,[e_j])\quad(j\in J)
\end{equation}
in $\C^I\oplus W_I$. These weights have the original equivalence
class, independently of $A$. Consequently one such support determines
the envelopes, restriction maps and monodromy at every other support.
\end{theorem}
\begin{proof}

The labelled quotient determines the monomials in its pullback algebra,
hence $S_I$, $M_I$, $R_I$ and $W_I$. The equality $\dim Q_I=n-k$ identifies
the residual weights with the classes of the source basis in $W_I$.
For each $\alpha\in S_I$, the character $X^\alpha$ is specified by
$q_I^*X^\alpha=w^\alpha$. Its eigenvalue under the labelled residue
$-\delta_i$ is $-c_i(\alpha)$ by Theorem~\ref{thm:connection}.
This recovers $C_I$ on $S_I$, then on $M_I$ by additivity and on
$R_I$ by complex linearity.

Let $\pi_I^{\mathrm{res}}:\C^J\to W_I$ be the quotient map, and
let $A,A'$ be extensions of $C_I^{\C}$. Their difference vanishes
on $R_I$, so $A'-A=\phi\circ\pi_I^{\mathrm{res}}$ for a unique
linear map $\phi:W_I\to\C^I$. The automorphism
$F_\phi(u,w)=(u+\phi(w),w)$ fixes every $\widehat\Lambda_i$,
$i\in I$, and carries the weights constructed from $A$ to those
constructed from $A'$. Thus \eqref{eq:reconstruct} has a
well-defined linear-equivalence class.

To identify this class with that of the original configuration, put $U_I=\Span\{\Lambda_i:i\in I\}\subset\g^*$.  Restriction gives the exact sequence $0\to U_I\to\g^*\to\h_I^*\to0$.  Choose a linear splitting $s:\h_I^*\to\g^*$.  Since the support weights form a basis of $U_I$, for every $j\in J$ write uniquely
\begin{equation*}
\Lambda_j=\sum_{i\in I}a_{ij}\Lambda_i+s(\Lambda_j|_{\h_I}),
\end{equation*}
and let $A_0:\C^J\to\C^I$ have columns $(a_{ij})_{i\in I}$.  For $\alpha\in M_I$,
\begin{equation*}
\sum_{j\in J}\alpha_j\Lambda_j
=\sum_{i\in I}(A_0\alpha)_i\Lambda_i
=\sum_{i\in I}c_i(\alpha)\Lambda_i.
\end{equation*}
Linear independence of the support weights gives $A_0|_{M_I}=C_I$; complex linearity then gives $A_0|_{R_I}=C_I^{\C}$.  Hence $A_0$ is an admissible extension.

\begin{samepage}
Let $\iota_I:\C^I\to U_I$ be the isomorphism $e_i\mapsto\Lambda_i$.  The equality $R_I=\ker\mu_I$ gives an isomorphism
$\overline\mu_I:W_I\to\h_I^*$. Together with the chosen splitting,
it defines
$\Psi:\C^I\oplus W_I\to\g^*$ by $\Psi(u,w)=\sum_{i\in I}u_i\Lambda_i+s(\overline\mu_I(w))$.  The reconstructed and original extension sequences are identified by
\begin{equation*}
\begin{tikzcd}[column sep=2.2em,row sep=1.8em]
0 \arrow[r]
& \C^I \arrow[r] \arrow[d,"\iota_I"',"\sim"]
& \C^I\oplus W_I \arrow[r] \arrow[d,"\Psi","\sim"']
& W_I \arrow[r] \arrow[d,"\overline\mu_I","\sim"']
& 0 \\
0 \arrow[r]
& U_I \arrow[r]
& \g^* \arrow[r]
& \h_I^* \arrow[r]
& 0.
\end{tikzcd}
\end{equation*}
\end{samepage}
For the configuration constructed from $A_0$, the identity $\Psi(\widehat\Lambda_\ell)=\Lambda_\ell$ holds for every $\ell$.  Independence of the extension proves the theorem.
\end{proof}

\begin{proposition}
\label{prop:reconstruction-failure}
There are ambient Poincar\'e-domain actions in linear general position whose labelled quotient maps, logarithmic connections and monodromy
at a singular coordinate stratum are identical, but whose labelled residual
configurations are not linearly equivalent. Their transverse
Camacho--Sad indices may be different.
\end{proposition}
\begin{proof}
Choose distinct positive real parameters $a,b$ and the residual
weight configurations $(1,a)$ and $(1,b)$ in dimension one. Each
has only the zero nonnegative resonance, so both labelled quotient
maps have target a point. By Theorem~\ref{thm:realisation}, both
occur at one-dimensional singular coordinate strata of rank-two ambient
actions in linear general position in the Poincar\'e domain. Since
the toric envelope is a point, its flat connection and monodromy are
necessarily trivial in either action. Nevertheless the residual
weights are distinct up to common nonzero scaling, and their
coordinate Camacho--Sad indices are respectively $a,a^{-1}$ and
$b,b^{-1}$, after labelling the two coordinates.
\end{proof}

\section{Classical residues and restriction maps}\label{sec:residues}
We compare residues of the original transverse foliation with residues
of fields induced by the envelope connection, using the diagonal
Baum--Bott formula and its relative versions
\cite{BaumBott,BSS,LehmannSuwa,CamachoLehmann}.

\subsection{The original foliation}
\begin{lemma}\label{lem:maximal-singular-component}
Let $I$ satisfy $|I|=k-1$.  Assume that the weights $\{\Lambda_i:i\in I\}$ are linearly independent and that $\Lambda_j|_{\h_I}\ne0$ for every $j\notin I$.  Put $J=I^c$ and $q=n-k$.  Then
$
Z_I=\{z_j=0:j\in J\}
$
is an irreducible component of $\operatorname{Sing}(\cF)$ of codimension $q+1$.  At every point of its dense stratum with support exactly $I$, a transversal is $\C^J$ and the induced foliation is the residual one-dimensional diagonal foliation of $\h_I$ with nonzero weights.
\end{lemma}

\begin{proof}
The support weights have rank $k-1$, hence every point with support contained in $I$ is singular and $Z_I\subset\operatorname{Sing}(\cF)$.  Since $\h_I$ is one-dimensional, the condition $\Lambda_j|_{\h_I}\ne0$ is equivalent to $\Lambda_j\notin\Span\{\Lambda_i:i\in I\}$.  Thus $I\cup\{j\}$ has rank $k$ for every $j\in J$.
For a diagonal action, a point is singular exactly when the weights indexed by its support have rank less than $k$.  Hence the reduced singular locus is the finite union of the coordinate subspaces corresponding to maximal rank-deficient supports.  The preceding rank calculation shows that $I$ is maximal among such supports, so $Z_I$ is an irreducible component.  Its codimension is $|J|=n-k+1=q+1$.  Proposition~\ref{prop:residual}, with $B=I$, gives the transverse action of the one-dimensional space $\h_I$ on $\C^J$; the residual weights are nonzero by hypothesis.
\end{proof}

Theorem~\ref{thm:ambiguity} identifies when the first-integral quotient determines the projective eigenvalues of this transverse singularity.  These projective eigenvalues enter directly into the Baum--Bott and Camacho--Sad invariants.  When $\dim Q_I=n-k$, the quotient determines every Baum--Bott coefficient considered here; in transverse dimension two the Camacho--Sad criterion is also necessary.  In higher transverse dimension an individual symmetric Baum--Bott residue can be coarser than the projective weights.

Let $I$ satisfy the hypotheses of Lemma~\ref{lem:maximal-singular-component}, put $J=I^c$ and $q=n-k$, and choose $0\ne\xi_I\in\h_I$.  Set $\lambda_j=\Lambda_j(\xi_I)$ for $j\in J$.  Since $\dim\h_I=1$, Proposition~\ref{prop:residual} gives the transverse field
\begin{equation}\label{eq:maximal-transverse-field}
Y_I=\sum_{j\in J}\lambda_j w_j\frac{\partial}{\partial w_j}.
\end{equation}
Every $\lambda_j$ is nonzero, so the transverse singularity is isolated.  At a generic point of $Z_I$, Baum--Bott localisation identifies the coefficient of $Z_I$ with the Grothendieck residue of this transverse one-dimensional foliation; it is independent of the generic point and of the chosen transversal~\cite{BaumBott}.  Denote it by $\operatorname{BB}_{\varphi}(\cF;Z_I)$.

\begin{proposition}\label{prop:originalres}
Under these hypotheses, for $q\ge1$ and every homogeneous symmetric
polynomial $\varphi$ of degree $q+1$,
$$
\operatorname{BB}_{\varphi}(\cF;Z_I)
=\frac{\varphi((\lambda_j)_{j\in J})}{\prod_{j\in J}\lambda_j}.
$$
When $\dim Q_I=q$, the labelled quotient map determines these coefficients. When $q=1$, with transverse coordinates $(x,y)$ and weights
$(\lambda_1,\lambda_2)$, the coordinate Camacho--Sad indices are
$\lambda_2/\lambda_1$ along $\{y=0\}$ and $\lambda_1/\lambda_2$
along $\{x=0\}$. The labelled quotient map determines this ordered pair
if and only if $\dim Q_I=1$.
\end{proposition}
\begin{proof}
By Lemma~\ref{lem:maximal-singular-component}, the transverse field
\eqref{eq:maximal-transverse-field} has an isolated zero with
derivative $\operatorname{diag}(\lambda_j)_{j\in J}$. The
Grothendieck formula gives the stated Baum--Bott coefficient
\cite{BaumBott,BSS}. Its numerator and denominator have the same
degree, so common scaling leaves it unchanged. If $\dim Q_I=q$, Theorem~\ref{thm:ambiguity} recovers the projective weights.

In transverse dimension two, the normal residue along $\{y=0\}$ is
\[
\operatorname{Res}_0\frac{\lambda_2\,dx}{\lambda_1x}
=\frac{\lambda_2}{\lambda_1};
\]
the other axis gives the inverse ratio. If $\dim Q_I=0$, then the
labelled residual configurations vary in the positive-dimensional
family of Theorem~\ref{thm:ambiguity}, so these ratios vary with the envelope fixed.
This proves necessity for the ordered Camacho--Sad pair.
\end{proof}

\subsection{Fields induced by the connection}
For any coordinate support $I\ne\varnothing$ with $r_I<k$, define
$\Theta_I:V_I\to\operatorname{Der}\C[S_I]$ by
$$\Theta_I(u)X^\alpha=\nu_\alpha(u)X^\alpha.$$
Additivity of $\nu$ gives commuting derivations. At independent
supports, $\Theta_I(e_i)=\delta_i$ for the dual support basis.
Differentiating transport equivariance gives
\begin{equation}\label{eq:Theta-support}
\Theta_{I'}(v')\operatorname{res}_{I,I'}
=\operatorname{res}_{I,I'}\Theta_I(\pi_{I',I}v').
\end{equation}

Assume $Q_I$ is smooth of dimension $m\ge1$ at $o_I$, so
$S_I\simeq\N^m$, and use the coordinates defined by its Hilbert basis
$x_a=X^{\beta^a}$. Write
$\lambda_a(u)=\nu_{\beta^a}(u)$; then
$\Theta_I(u)=\sum_a\lambda_a(u)x_a\partial_{x_a}$.

\begin{proposition}\label{prop:inducedres}
If every $\lambda_a(u)\ne0$, then the induced field has an isolated zero
and, for $\deg\varphi=m$,
$$
\operatorname{Res}_{\varphi}(\Theta_I(u),o_I)
=\frac{\varphi(\lambda_1(u),\ldots,\lambda_m(u))}
       {\lambda_1(u)\cdots\lambda_m(u)}.
$$
For $m=2$ the coordinate Camacho--Sad indices are
$\lambda_2(u)/\lambda_1(u)$ and its inverse. At independent supports
$\lambda_a(u)=\sum_i u_i c_i(\beta^a)$, so logarithmic residues
determine these quantities.

On a support coordinate disc with coordinate $t=z_i$, the field
$Z_{I,i}=t\partial_t+\delta_i$ has horizontal leaves over $t\ne0$
and monodromy $X^\alpha\mapsto e^{2\pi i c_i(\alpha)}X^\alpha$.
If all $c_{ia}=c_i(\beta^a)$ are nonzero, then its isolated zero $(0,o_I)$
has, for $\deg\psi=m+1$, residue
$$
\operatorname{Res}_\psi(Z_{I,i},(0,o_I))
=\frac{\psi(1,c_{i1},\ldots,c_{im})}{\prod_a c_{ia}}.
$$
For $m=1$ the index along $\{x_1=0\}$ is $c_{i1}$, with holonomy
multiplier $e^{2\pi i c_{i1}}$.
\end{proposition}
\begin{proof}
In these coordinates the induced field is diagonal, with eigenvalues
$\lambda_a(u)$. The isolated-zero formula gives its Baum--Bott
residue. In dimension two, the normal and tangent eigenvalues along
$\{x_2=0\}$ are $\lambda_2(u)$ and $\lambda_1(u)$, giving the
stated Camacho--Sad ratio; the other axis reverses them.

Along the $i$th support disc, horizontality means
$dx_a=c_i(\beta^a)x_a\,dt/t$. Thus $Z_{I,i}$ has
eigenvalues $1,c_{i1},\ldots,c_{im}$, and the same
isolated-zero formula gives its residue. For $m=1$, the tangent
eigenvalue is $1$ and the normal eigenvalue is $c_i(\beta^1)$.
Integrating the horizontal equation around $t=0$ gives multiplier
$e^{2\pi i c_i(\beta^1)}$.
\end{proof}

Monodromy loses the integral part of these coefficients. For
$\operatorname{Re}a,\operatorname{Re}b>0$ and $N\in\N$, take weights
$(1,0),(a+N,1),(b,-1)$. All are Poincar\'e-domain configurations
in linear general position: the only remaining determinant is
$-(a+b+N)\ne0$. At the first support, $Q=\C$ with coordinate $xy$,
the connection coefficient is $a+b+N$, and monodromy is always
$e^{2\pi i(a+b)}$. The Camacho--Sad indices at $(0,o_I)$
are distinct, although the envelope and monodromy action groupoid
are unchanged.

\subsection{Singular toric quotients}
Choose a Hilbert basis $\beta^1,\ldots,\beta^N$ and write
$\C[S_I]=\C[x_1,\ldots,x_N]/J_I$, $x_a\mapsto X^{\beta^a}$.
For $u\in V_I$, put
$\widetilde\Theta_I(u)=\sum_a\nu_{\beta^a}(u)x_a\partial_{x_a}$.

\begin{proposition}
\label{prop:lci}
The diagonal field $\widetilde\Theta_I(u)$ preserves the toric
ideal $J_I$ and induces $\Theta_I(u)$ on $Q_I$. For each toric
binomial $f_{p,q}=x^p-x^q$ one has
\begin{equation}\label{eq:toric-eigenrelation}
\widetilde\Theta_I(u)(f_{p,q})
=\kappa_{p,q}(u)f_{p,q},\qquad
\kappa_{p,q}(u)=\sum_a p_a\nu_{\beta^a}(u)
=\sum_a q_a\nu_{\beta^a}(u).
\end{equation}
Suppose in addition that $Q_I$ is a local complete intersection
at $o_I$, and choose a regular sequence $f=(f_1,\ldots,f_c)$
for the toric ideal in a local ambient neighbourhood. There is a
holomorphic matrix $K_u$ with
$\widetilde\Theta_I(u)f=K_uf$; its class modulo $J_I$ gives
the connection along $\Theta_I(u)$ on the conormal module
$J_I/J_I^2$ induced by the field. Since the distinguished vertex is
fixed, its action on the conormal fibre at $o_I$ is a linear
endomorphism. The corresponding virtual tangent class is
\begin{equation}\label{eq:virtual-tangent-class}
\tau_{Q_I}=[T_{\C^N}|_{Q_I}]-[(J_I/J_I^2)^\vee]
\quad\text{in }K^0(Q_I).
\end{equation}
When the ideal is generated by an eigenvector regular sequence,
its conormal fibre has the corresponding eigenvalues
$\kappa_{p,q}(u)$, in the convention for the induced derivation on
functions. At independent supports, they are linear combinations of the
logarithmic residue coefficients $c_i(\beta^a)$.
\end{proposition}
\begin{proof}
A binomial $f_{p,q}$ is in $J_I$ precisely when
$\sum_a p_a\beta^a=\sum_a q_a\beta^a$. Applying the
additive functional $\nu_{(-)}(u)$ yields
$\sum_a p_a\nu_{\beta^a}(u)
=\sum_a q_a\nu_{\beta^a}(u)$.
Differentiation gives~\eqref{eq:toric-eigenrelation} and proves
ideal invariance; the induced action on every $X^{\beta^a}$
is the defining action of $\Theta_I(u)$.
For a regular sequence generating the local ideal, invariance gives
$\widetilde\Theta_I(u)(f_i)\in(f_1,\ldots,f_c)$,
so the coefficients define $K_u$. The formula $\widetilde\Theta(af)=\widetilde\Theta(a)f
+a\widetilde\Theta(f)$ shows that this is a connection along
the induced derivation, not an $\mathcal O_{Q_I}$-linear map on
$J_I/J_I^2$. On the fibre at the fixed vertex the Leibniz term
vanishes modulo the maximal ideal, leaving an intrinsic linear
endomorphism. Changing the chosen generators conjugates its matrix
on this fibre. The dual of the conormal presentation gives the perfect
tangent complex $[T_{\C^N}|_{Q_I}\to(J_I/J_I^2)^\vee]$,
whose virtual class is the difference in~\eqref{eq:virtual-tangent-class}.
Finally, at independent supports,
$\nu_{\beta^a}(u)=\sum_i u_i c_i(\beta^a)$,
which proves the last assertion.
\end{proof}

This determines the infinitesimal action on the virtual tangent bundle.
To obtain a numerical residue, one must also specify a characteristic
polynomial and verify the isolated-singularity hypotheses in
\cite{LehmannSuwa,BSS}.

For example, on the quadratic cone
$Q=\{ub-va=0\}\subset\C^4$, consider the diagonal field
$u\partial_u+2v\partial_v+3a\partial_a+4b\partial_b$.
The linear field integrates to a $\C^*$-action. Both monomials in the
defining equation have weight $5$, so the normal representation has
weight $5$, while the ambient tangent representation has weights
$1,2,3,4$. The equivariant total Chern class of the virtual tangent
representation is therefore
\[
\frac{(1+t)(1+2t)(1+3t)(1+4t)}{1+5t}.
\]
Here $t\in H^2(B\C^*;\Z)$ is the first Chern class of the weight-one
representation.

Return to $I\subset I'$, with $D=I'\setminus I$ and
$B=B_{I,I'}=S_I\cap\N^{(I')^c}$. The common algebra $\C[B]$ is both
a quotient of $\C[S_I]$ and a subalgebra of $\C[S_{I'}]$.

\begin{proposition}\label{prop:common}
For independent $I,I'$ and $\alpha\in S_{I'}$, one has
$\alpha\in B$ exactly when $c_d^{I'}(\alpha)=0$ for all $d\in D$;
then $c_i^{I'}(\alpha)=c_i^I(\alpha)$ for $i\in I$. Moreover
$$
\dim\Theta_{I'}(\ker\pi_{I',I})
=\dim\Span\{(c_d^{I'}(\alpha))_{d\in D}:\alpha\in S_{I'}\}.
$$
Thus $S_{I'}=B$ precisely when transport in the additional support
directions is trivial, or equivalently when every
$\delta_d^{I'}=0$ for $d\in D$.

For coordinate supports $\varnothing\ne I\subset I'$ with $r_I,r_{I'}<k$, the field
$X^\alpha\mapsto\nu_\alpha(v')X^\alpha$ on $\C[B]$ is both the
quotient of $\Theta_I(\pi v')$ and the restriction of
$\Theta_{I'}(v')$. If the common toric quotient $\Spec\C[B]$ is smooth of dimension $m\ge1$
and the eigenvalues $\lambda_a$ in Hilbert-basis coordinates are nonzero, then for
$\deg\varphi=m$ its residue from either support is
\[
\frac{\varphi(\lambda_1,\ldots,\lambda_m)}
{\lambda_1\cdots\lambda_m}.
\]
For $m=2$, the coordinate Camacho--Sad indices are
$\lambda_2/\lambda_1$ and $\lambda_1/\lambda_2$.
\end{proposition}
\begin{proof}
Independence makes the expansion of $\nu_\alpha$ in the $I'$-weights
unique. It belongs to the span of the $I$-weights exactly when its
$D$-coefficients vanish, which is precisely the resonance condition
at $I$. The displayed dimension is the rank of the pairing between
the additional support directions and these coefficients. For arbitrary
supports, the restriction kernel is a monomial ideal, stable under
the diagonal derivations. Equation~\eqref{eq:Theta-support} identifies
the two induced fields on $\C[B]$. The stated residue formulas then
apply to this same field.
\end{proof}

\section{Nonlinear normal forms and induced derivations}\label{sec:nonlinear}
We turn to the effect of resonant nonlinear terms on the first-integral
algebra, assuming a \emph{convergent} simultaneous Poincar\'e--Dulac normal form
on $(\C^s,0)$:
\begin{equation}\label{eq:normalform}
X_a=S_a+R_a,\qquad S_a=\sum_{j=1}^s\mu_{aj}w_j\partial_{w_j},
\qquad 1\le a\le d,
\end{equation}
where $(\mu_{aj})$ has rank $d$, all coordinate coefficients of $R_a$
vanish to order at least two, and
$[X_a,X_b]=[S_a,R_b]=0$ for all $a,b$.
Both families have generic rank $d$, since the lowest-degree term of
$X_1\wedge\cdots\wedge X_d$ is $S_1\wedge\cdots\wedge S_d\ne0$.
Put $\kappa=s-d$. Let $S\subset\N^s$ be the joint resonance semigroup
and $q_0:(\C^s,0)\to(Q,o)$ its toric envelope. In this section $Q$
is the envelope of the semisimple system; we call it the
\emph{semisimple envelope}. For comparison, a resonant vector-field term
$w^\alpha\partial_{w_j}$ satisfies
$\sum_l\alpha_l\mu_{al}=\mu_{aj}$ for every $a$, whereas $S$ indexes
the monomials of joint weight zero.

\begin{theorem}\label{thm:nonlinear}
Each $R_a$ induces a unique holomorphic derivation $\overline R_a$ of
$\cO_{Q,o}$, with $R_aq_0^*=q_0^*\overline R_a$. These derivations
commute, and
\begin{equation}\label{eq:nonlinearalgebra}
\cO_{\C^s,0}^{X_1,\ldots,X_d}
=q_0^*\left(\bigcap_a\ker\overline R_a\right).
\end{equation}
If the induced foliation on $Q$ has an analytic envelope, then its composition
with $q_0$ is an envelope of the original nonlinear foliation.
\end{theorem}
\begin{proof}
Let $\mathcal A=\cO_{\C^s,0}^{S_1,\ldots,S_d}$. By the
diagonal envelope theorem, $\mathcal A=q_0^*\cO_{Q,o}$.
For every $a,b$, the commutation $[S_b,R_a]=0$ implies
$S_b(R_af)=R_a(S_bf)$, so $R_a$ preserves $\mathcal A$.
Because the pullback $q_0^*$ is injective, there is a unique
$\C$-linear operator $\overline R_a$ on $\cO_{Q,o}$ with
$R_aq_0^*=q_0^*\overline R_a$. It satisfies the Leibniz rule,
since $R_a$ does. On the finite Hilbert generating set its
coefficients are holomorphic, and the toric relations are preserved
by the derivation; hence it defines a holomorphic vector field on the analytic germ $Q$. The fields commute: expanding
$[X_a,X_b]=0$ and using $[S_a,S_b]=[S_a,R_b]=0$ gives
$[R_a,R_b]=0$, and injectivity transfers this equality to the
induced operators.

It remains to prove that every joint first integral of the $X_a$
is semisimple-invariant. On formal power series write
$\omega(\alpha)=(\sum_j\alpha_j\mu_{aj})_{a=1}^d$ for the
joint weight of a monomial. Each homogeneous degree contains
finitely many monomials, so it decomposes into finite-dimensional
joint weight spaces. Since $[S_b,R_a]=0$, the operator $R_a$
preserves joint weights; its coefficients vanish to order at least
two, so it raises the degree of a non-constant monomial by at least
one. The constant term causes no exception, since every $R_a$
annihilates constants and constants have joint weight zero. Suppose a first integral had a nonzero formal weight component
$f_\omega$ for $\omega\ne0$, and let $\ell$ be the least degree
at which that component appears. Choose $a$ with $\omega_a\ne0$.
In the weight-$\omega$ component of $X_af=0$ at degree $\ell$,
$R_a$ makes no contribution, whereas $S_a$ acts by scalar
$\omega_a$. Thus $\omega_a(f_\omega)_\ell=0$, a contradiction.
All the nonzero Taylor coefficients therefore have weight zero.
The convergent Taylor series of $f$ is supported on $S$, so the
analytic envelope theorem gives $f=q_0^*g$ for a unique
$g\in\cO_{Q,o}$. Now
$X_af=S_aq_0^*g+R_aq_0^*g=q_0^*\overline R_ag$,
since $S_a$ annihilates $q_0^*g$. Injectivity yields the asserted
common-kernel identity. If the kernel itself has a representing
analytic envelope, then the universal factorisation property applied
twice gives the final assertion.
\end{proof}

\begin{corollary}\label{cor:nonlinear}
Let $\ell=\dim Q$ and let $r_{\mathrm{red}}$ be the generic rank of
the reduced distribution on $Q_{\mathrm{reg}}$. The number of independent
holomorphic first integrals is at most $\ell-r_{\mathrm{red}}\le\kappa$.
It equals $\kappa$ if and only if $\ell=\kappa$ and every
$\overline R_a=0$. In this case $q_0$ is also the nonlinear envelope.
\end{corollary}
\begin{proof}
On the regular locus, first-integral differentials annihilate a
distribution of rank $r_{\mathrm{red}}$. The dense-torus submersion
of $q_0$ preserves functional independence, and $\ell\le s-d$.
Equality forces $r_{\mathrm{red}}=0$, hence each reduced derivation
vanishes on a dense open set and therefore everywhere. Conversely,
vanishing identifies the invariant algebra with $\cO_{Q,o}$, and
$\ell$ independent resonance monomials give the required rank.
\end{proof}

The fields $R_a$ may remain nonzero when tangent to every fibre of
$q_0$. For normalisation and integrability under other hypotheses,
see Stolovitch~\cite{StolovitchSCI,Stolovitch} and
Zung~\cite{ZungNormalForms}.

\subsection{Explicit losses of first integrals}
For the saddle
$$
X_\varepsilon=x\partial_x-y\partial_y-\varepsilon xy^2\partial_y,
\qquad u=xy,
$$
the semisimple envelope is $\C_u$ and the induced nonlinear field is
$-\varepsilon u^2\partial_u$. Hence its first-integral algebra is
$\C\{u\}$ for $\varepsilon=0$ and $\C$ otherwise.

A singular semisimple envelope can also have a non-trivial smooth
nonlinear envelope. On $\C^4$ take
\begin{equation}\label{eq:quadric}
S_0=x_1\partial_{x_1}+x_2\partial_{x_2}
-y_1\partial_{y_1}-y_2\partial_{y_2},\qquad
R=(x_1y_1)x_2\partial_{x_2},\qquad X=S_0+R.
\end{equation}
The semisimple envelope is
$Q=\{ub-va=0\}\subset\C^4$, with
$(u,v,a,b)=(x_1y_1,x_1y_2,x_2y_1,x_2y_2)$.
The reduced field fixes $u,v$ and sends $a,b$ to $ua,ub$.
The invariant algebra of the reduced field is exactly $\C\{u,v\}$:
the equation $Rf=0$ implies $x_2\partial_{x_2}f=0$ in the domain
$\C\{x_1,x_2,y_1,y_2\}$, so $f$ is independent of $x_2$; the remaining
semisimple invariant monomials are generated by $u,v$.
Thus $(x_1,x_2,y_1,y_2)\mapsto(u,v)$ is the nonlinear envelope and
has functional rank two, whereas $\dim Q=3$.

\begin{proposition}\label{prop:jets}
For every $d\ge1$, jet order $J\ge1$ and $0\le b\le d$, there are
commuting polynomial normal forms on $\C^{2d}$ agreeing to order
$J$ with a diagonal system admitting $d$ independent integrals, but
admitting exactly $d-b$ independent holomorphic integrals themselves.
Their nonlinear analytic envelope is $\C^{d-b}$.
\end{proposition}
\begin{proof}
Choose $N$ with $2N+1>J$ and coefficients
$\varepsilon_1,\ldots,\varepsilon_d$, exactly $b$ of which
are nonzero. Let
\[
X_a=x_a\partial_{x_a}-y_a\partial_{y_a}
-\varepsilon_a x_a^Ny_a^{N+1}\partial_{y_a}.
\]
The fields act on disjoint coordinate pairs and therefore
commute. They differ from the diagonal fields in terms of total
polynomial degree $2N+1$, so all their Taylor jets of order $J$
are identical. Their semisimple first-integral algebra is
$\C\{u_1,\ldots,u_d\}$, $u_a=x_ay_a$, because an invariant
monomial has equal $x_a$- and $y_a$-exponents in each pair.
A direct calculation gives
$X_a(u_c)=-\delta_{ac}\varepsilon_a u_a^{N+1}$; hence the
induced fields on the semisimple envelope are
$-\varepsilon_a u_a^{N+1}\partial_{u_a}$.
A convergent series is annihilated by each such derivation if and
only if its partial derivative with respect to every $u_a$ with
$\varepsilon_a\ne0$ vanishes: the convergent power series ring is
an integral domain. Therefore the common-kernel algebra is exactly
\[
\C\{u_a:\varepsilon_a=0\},
\]
which is a convergent power series algebra in $d-b$ variables.
The corresponding monomials are independent on a dense torus and
no further independent integrals exist by the nonlinear reduction
theorem.
\end{proof}

For a fixed normal form, the reduced derivations vanish if and only if
$R_a(w^{\beta^j})=0$ for every $a,j$, where the $\beta^j$ form a
Hilbert basis. These are identities of holomorphic germs.

\subsection{Semiglobal first integrals and holonomy}
Let $L$ be a connected complex manifold and suppose a germ of a foliated
neighbourhood along $L$ has product charts
$(U_\nu\times(\C^s,0),TU_\nu\oplus\cF_X)$ with locally constant
transverse transition germs fixing zero. Assume these germs preserve
both the nonlinear and semisimple foliations of \eqref{eq:normalform}.
They define the holonomy representation
$h:\pi_1(L)\to\Aut(\C^s,0)$ and, by the universal property of the
semisimple envelope, $\bar h:\pi_1(L)\to\Aut(Q,o)$.

\begin{theorem}\label{thm:nonlinearsg}
Restriction to a fixed transversal identifies the semiglobal
first-integral algebra with
$$
q_0^*\left[\left(\bigcap_a\ker\overline R_a\right)
                 ^{\bar h(\pi_1(L))}\right].
$$
If $\bar h$ takes values in the dense torus of $Q$, then there are
$\kappa=s-d$ independent semiglobal integrals precisely when
$\dim Q=\kappa$, every $\overline R_a=0$, and $\operatorname{im}\bar h$
is finite. In this case the semiglobal envelope is
$Q/\operatorname{im}\bar h$.
\end{theorem}
\begin{proof}
In each product chart a first integral is independent of the base
and belongs to the transverse algebra of Theorem~\ref{thm:nonlinear}.
These germs glue along $L$ exactly when holonomy fixes them, proving
the formula. Conversely an invariant germ transports consistently;
representatives may be shrunk transversely in the given germ atlas.

Complete semiglobal integrability first implies complete local
integrability, so Corollary~\ref{cor:nonlinear} forces $\dim Q=\kappa$ and vanishing of all reduced derivations. For toric holonomy put
$M^h=\{\alpha\in M:\chi_\alpha(\bar h_\gamma)=1\ \forall\gamma\}$
and $S^h=S\cap M^h$. Comparing coefficients after pullback to the
source, and applying the Hilbert-basis convergence argument, identifies
the invariant analytic algebra with that of $\Spec\C[S^h]$.
Full functional rank implies $\operatorname{gp}(S^h)$, hence $M^h$,
has finite index in $M$. Holonomy is then contained in the finite
group $\Hom(M/M^h,\C^*)$. Conversely a finite image of exponent $N$
gives $NS\subset S^h$, hence full rank and the asserted finite quotient.
\end{proof}

For an explicit suspension, use \eqref{eq:quadric} and
$h(x_1,x_2,y_1,y_2)=(x_1,2x_2,y_1,3y_2)$. It preserves $S_0$ and
$R$, since $u=x_1y_1$ is fixed. On
$$
\mathcal M=(\C_t\times\C^4)/\Z,\qquad
n(t,w)=(t+n,h^nw),
$$
the action is free and properly discontinuous by translation in $t$.
The fields $\partial_t,X$ descend and have generic rank two, with
singular leaf $L=\C/\Z\simeq\C^*$. The holonomy group is generated
on $Q$ by $(u,v,a,b)\mapsto(u,3v,2a,6b)$ and on the nonlinear
envelope by $(u,v)\mapsto(u,3v)$. Therefore
$$
\mathcal B_L\simeq\C\{u,v\}^{\langle(u,v)\mapsto(u,3v)\rangle}
=\C\{u\}.
$$
The semisimple, nonlinear local and semiglobal functional ranks are
respectively $3,2,1$.

\section{Applications to Baum--Bott residues}
\label{sec:bb-application}

The compact-component calculation uses the reduction of the transverse
linear weights to proportional integers, comparison with the normal model
by deformation and continuity of residues \cite{BracciSuwa}, and the
classical Baum--Bott formula \cite[Section~11]{BaumBott}. The comparison
applies to the nonlinear foliation under the product-type assumptions
stated below.
For affine diagonal actions, a separate equivariant argument proves
vanishing in higher degree; reconstruction connects it to actions of
complex vector groups. All localised classes are those of the Baum--Bott
construction.
In this section the tangent sheaf is saturated: its quotient in the
ambient tangent sheaf is torsion-free. Saturation may change the orbit
partition used above when an orbit rank drops along a divisor.

Let $B$ be a smooth compact connected component of the singular set
of a saturated holomorphic foliation $\cF$ of codimension $q\ge1$
on a complex $n$-manifold. Put $r=q+1$ and $m=n-r$. Suppose that a neighbourhood $U$ of $B$ has foliated charts
$(z,w)\in\C^m\times\C^r$ with $B=\{w=0\}$ and
\begin{equation}\label{eq:bb-product-type}
\begin{gathered}
T\cF=\left\langle
\partial_{z_1},\ldots,\partial_{z_m},X_a(w)\right\rangle_{\cO},
\qquad X_a(w)=Y(w)+O(|w|^2),\\
Y=\sum_{i=1}^r\lambda_iw_i\partial_{w_i},\qquad \lambda_i\ne0.
\end{gathered}
\end{equation}
The transverse fields are independent of $z$ and share the same
invertible semisimple linear part. The higher-order terms may be
nonlinear, as may transitions between foliated product charts, whose
base coordinates may also depend on $w$.

Write $S=\{\alpha\in\N^r:\sum_i\alpha_i\lambda_i=0\}$ and assume
\begin{equation}\label{eq:bb-full-rank}
\rk\operatorname{gp}(S)=r-1.
\end{equation}
Let $q_0:(\C^r,0)\to(Q,o)$ be the toric envelope of the common
linear part $Y$. We use the complex orientation for Thom and
Poincar\'e duality. For a homogeneous symmetric
characteristic polynomial $\varphi$ of degree $d>q$, denote the
distinguished localisation by
\[
\alpha_{\varphi,B}(\cF)\in H^{2d}(U,U\setminus B;\C).
\]
Its Alexander dual is $\operatorname{Res}_{\varphi}(\cF,B)$.
Under the Thom isomorphism it corresponds to the Poincar\'e dual on
$B$, of degree $2(d-r)$. We compute this distinguished class
\cite{CorreaSuwa}.

\subsection{Integral weights and the normal bundle}

\begin{lemma}\label{lem:bb-integral-weights}
Condition~\eqref{eq:bb-full-rank} implies that
$\lambda_i=t n_i$, where $t\in\C^*$ and
$n=(n_1,\ldots,n_r)\in\Z^r$ is primitive, every $n_i$ is nonzero,
and $n$ has both signs. Moreover
\[
\operatorname{gp}(S)=\ker(n:\Z^r\to\Z),
\]
and $S$ contains an exponent all of whose coordinates are strictly
positive. Multiplying every $X_a$ by $t^{-1}$ normalises its
linear part to $\sum_i n_iw_i\partial_{w_i}$ without changing $\cF$.
\end{lemma}
\begin{proof}
The complex relation space of the $\lambda_i$ has dimension $r-1$.
By \eqref{eq:bb-full-rank} it is spanned by integral resonances,
so its annihilator is the complex line spanned by a primitive
integral vector $n$. Thus $\lambda=t n$. If all $n_i$
had the same sign, the nonnegative resonance semigroup would be
zero, contrary to $r\ge2$ and \eqref{eq:bb-full-rank}.
Proposition~\ref{prop:rank-one-signs} gives $\operatorname{gp}(S)=\ker n$;
its proof constructs a strictly positive resonance.
\end{proof}

Henceforth $Y=\sum_i n_iw_i\partial_{w_i}$ and
$D=\operatorname{diag}(n_1,\ldots,n_r)$.

\begin{proposition}\label{prop:bb-transition-reduction}
Under \eqref{eq:bb-product-type} and \eqref{eq:bb-full-rank}, there
is a connected finite \'etale covering $p:U'\to U$, of degree at
most two, for which the normal transition matrices $A_{ba}$ satisfy
\begin{equation}\label{eq:bb-projective-flatness}
A_{ba}D=DA_{ba},\qquad
(dA_{ba})A_{ba}^{-1}=D\theta_{ba},
\end{equation}
where $\theta$ is a closed holomorphic one-form cocycle on
$B'=p^{-1}B$. No condition is imposed on envelope holonomy or on
the higher transverse terms of the foliated transitions.
\end{proposition}
\begin{proof}
Write a transition from chart $a$ to chart $b$ as
$(F(z,w),H(z,w))$, and put $h_z(w)=H(z,w)$ and
$A(z)=\partial_wH(z,0)$. Preservation of the tangent sheaves gives
$dh_z\,X_a=uX_b\circ h_z$, where $u$ is a holomorphic unit.
Indeed the coefficient and its inverse are holomorphic off $w=0$
and extend by Hartogs, since $r\ge2$.
Comparing linear terms gives $AD=cDA$, where $c=u(z,0)$.
The scalar $c$ permutes the finite multiset of nonzero integer
eigenvalues, so $c\in\{1,-1\}$. It is locally constant and its
transition values form a cocycle. Kill this cocycle on its monodromy
covering, then change frames and rescale generators by constants
so that $AD=DA$ on every overlap.

This covering of $B$ extends to a finite covering of a smooth
tubular neighbourhood, which retracts onto $B$. The unique complex
structure pulled back from the neighbourhood makes it finite
\'etale. Its degree is one or two.

For each base coordinate, preservation of $T\cF$ gives
$\partial_{z_j}H=a_j(z,w)X_b(H)$ with a holomorphic coefficient $a_j$.
Its term linear in $w$ is
$\partial_{z_j}A=a_j(z,0)DA$. Hence
$(dA)A^{-1}=D\theta$, where
$\theta=\sum_j a_j(z,0)dz_j$.
The Maurer--Cartan identity and $\theta\wedge\theta=0$ imply
$D\,d\theta=0$, and $D$ is invertible; thus $d\theta=0$.
The normal-matrix cocycle identity gives the additive cocycle
identity for $\theta$.
\end{proof}

\subsection{Baum--Bott localisation and the normal bundle}

\begin{lemma}\label{lem:bb-normal-comparison}
Let $B$ be a smooth compact connected singular component of
codimension $r\ge2$, with local product charts generated by
$\partial_z$ and $X_a(w)$, where every $dX_a(0)$ is invertible.
The base directions and the linear fields $dX_a(0)w$ define a
foliation $\cF_E$ on $E=N_{B/U}$.
Under Thom identification with $H^{2(d-r)}(B;\C)$,
\begin{equation}\label{eq:bb-canonical-comparison}
\alpha_{\varphi,B}(\cF)
=\alpha_{\varphi,B}(\cF_E),\qquad r\le d\le n.
\end{equation}
\end{lemma}
\begin{proof}
The deformation to the normal bundle has transitions
\begin{equation}\label{eq:bb-deformation}
(z,w,t)\longmapsto
\bigl(F_{ba}(z,tw),\,t^{-1}H_{ba}(z,tw),\,t\bigr).
\end{equation}
The transverse expression extends holomorphically to $t=0$ as
$\partial_wH_{ba}(z,0)w$. The cocycle and inverse identities
extend from $t\ne0$, where this is conjugation by
$R_t(z,w)=(z,tw)$. Compactness of $B$ permits a finite atlas
refinement and shrinking of the transverse domains, giving a
holomorphic family $\pi:\mathfrak U\to\Delta$ near $B\times\Delta$,
with $\Delta$ a bounded disc containing $0$ and $1$.

The local transverse generators are
$$
X_{a,t}(w)=t^{-1}X_a(tw),\qquad X_{a,0}(w)=dX_a(0)w.
$$
Write $X_{a,t}(w)=M_a(t,w)w$ with $M_a$ holomorphic and
$M_a(t,0)=dX_a(0)$. After shrinking, these matrices are invertible
uniformly on a relatively compact parameter disc. Hence the
transverse zero set is exactly $w=0$ throughout the family.
The modules generated by $\partial_z$ and $X_{a,t}$ are saturated.
Indeed, multiplication by the invertible matrix $M_a(t,w)^{-1}$
identifies the transverse quotient with
$\cO^r/\cO(w_1,\ldots,w_r)$; this module is torsion-free for
$r\ge2$. For $t\ne0$ the transition maps preserve the modules
by conjugation. Their tangency identities are holomorphic in $t$,
and therefore extend to $t=0$; applying the same argument to inverse
transitions gives equality of the modules on overlaps. They define
an involutive locally free tangent sheaf
$\mathscr V\subset T_{\mathfrak U/\Delta}$, with singular set
$B\times\Delta$ and central fibre $\cF_E$.

The presentation
$$
0\longrightarrow\mathscr V\longrightarrow
T_{\mathfrak U/\Delta}\longrightarrow N^{\mathrm{rel}}\longrightarrow0
$$
remains exact on every fibre: after cancelling the base directions,
the remaining map is $a\mapsto aX_{a,t}$, which is injective because
$X_{a,t}$ is nonzero away from $w=0$. The cokernel is flat over
$\Delta$: the local flatness criterion applies to this free
presentation, since the map remains injective after restriction to
each parameter fibre. The corresponding smooth presentation is also
exact: any smooth section in its kernel vanishes off
$B\times\Delta$, and hence everywhere. Thus the family has the
smooth locally free resolution required by
\cite[Remark~5.1 and Theorem~5.4]{BracciSuwa}.
The continuity theorem applies for $d>q=r-1$.

Choose a smooth tubular retraction onto $B\times\Delta$ over
$\Delta$, and compose it with projection to $B$. Pulling back any
closed form on $B$ gives a smooth form whose fibre restrictions
are closed and represent the same class under the tubular
identifications. The cited theorem therefore makes the pairings
of $\operatorname{Res}_{\varphi}(\cF_t,B)$ with every class in
$H^{2(n-d)}(B;\C)$ continuous in $t$, using the fixed
identification of the singular loci with $B$. For $t\ne0$ the
foliated isomorphism $R_t$ restricts to the identity on $B$;
consequently these pairings are independent of $t\ne0$ and equal
to those for $\cF$. Their continuity at $0$ gives equality of
all pairings for $\cF_E$ and $\cF$. Since $B$ is compact and
oriented, these pairings determine the residue class. Alexander
duality and the Thom isomorphism now give
\eqref{eq:bb-canonical-comparison}.
\end{proof}

\begin{lemma}
\label{lem:bb-torus-suspension}
Let $G=(\C^*)^k$ act holomorphically on a complex manifold $X$,
tangentially to a saturated foliation $\mathcal H$ of codimension $q$
whose normal sheaf has a $G$-equivariant locally free resolution.
Let $S$ be a closed $G$-invariant analytic subset containing
$\operatorname{Sing}(\mathcal H)$. For a homogeneous characteristic
polynomial $\varphi$ of degree $d>q$, there is a class
\[
\alpha^G_{\varphi,S}\in H_G^{2d}(X,X\setminus S;\C),
\]
whose ordinary image is the distinguished Baum--Bott class and
whose absolute image is $\varphi_G(N\mathcal H)$.
\end{lemma}
\begin{proof}
Write $E_a=(\C^{a+1}\setminus0)^k$, $B_a=(\mathbb P^a)^k$ and
$X_a=E_a\times^G X$, with $S_a=E_a\times^G S$. A local section of
$E_a\to B_a$ identifies $X_a$ with a product of an open subset of
$B_a$ and $X$. In this chart take the foliation generated by the
base directions and $T\mathcal H$. Under a change of section, the
derivative in a base direction acquires a fundamental vector field
of the $G$-action. This field belongs to $T\mathcal H$ because the
action is tangent to the foliation. Thus the local distributions
glue to a holomorphic foliation $\mathcal H_a$ of codimension $q$,
whose singular set is contained in $S_a$.
The normal sheaf and its locally free resolution are obtained by
gluing the pullbacks of those for $\mathcal H$ with the prescribed
$G$-equivariant transition maps.

The Baum--Bott construction can be carried out on the pair
$(X_a,X_a\setminus S_a)$. Its defining connections may be chosen
compatibly with the product descriptions. On the regular locus,
the Bott partial connections of $\mathcal H_a$ restrict to those
of $\mathcal H$ along a fibre. Both a fibre inclusion and an
inclusion $X_a\hookrightarrow X_{a+1}$ are locally product
inclusions; the normal resolutions remain exact after restriction.
Naturality of the relative \v{C}ech--de~Rham representatives
therefore identifies their \emph{distinguished} localised classes,
not merely their absolute characteristic classes. This is also the
product case of the functoriality construction in
\cite{CorreaSuwa}; the direct argument applies in codimension one
as well.

For fixed $d$, the spaces $E_a$ are sufficiently highly connected
when $a>d$ to compute the Borel relative cohomology in degree $2d$.
The compatible Baum--Bott classes consequently determine
$\alpha^G_{\varphi,S}\in H_G^{2d}(X,X\setminus S;\C)$.
Restriction along a fibre $X\hookrightarrow X_a$ gives the
ordinary distinguished localisation. On forgetting the support,
the class is the equivariant characteristic class of the descended
normal resolution, namely $\varphi_G(N\mathcal H)$.
\end{proof}

\begin{lemma}
\label{lem:bb-model-formula}
Let a rank-$r$ bundle $E\to B'$ have transition matrices satisfying
\eqref{eq:bb-projective-flatness}. Let
$E=\bigoplus_{\nu}E_\nu$ be its decomposition into the distinct
integer weight blocks of $D$, and put $s_\nu=\rk E_\nu$.
There is a common rational class
\begin{equation}\label{eq:bb-common-class}
\xi=\frac{c_1(E_\nu)}{s_\nu\nu}\in H^2(B';\Q),
\end{equation}
independent of $\nu$. For every $r\le d\le n$,
\begin{equation}\label{eq:bb-normal-formula}
\operatorname{PD}_{B'}\operatorname{Res}_{\varphi}(\cF_E,B')
=\frac{\varphi(n_1,\ldots,n_r,0,\ldots,0)}{n_1\cdots n_r}
\,\xi^{d-r}.
\end{equation}
Here evaluation of $\varphi$ has its stable characteristic-class
meaning; the trivial summand in $E-\cO$ contributes a zero root.
\end{lemma}
\begin{proof}
Solve the additive transition equation for the closed cocycle
$\theta$ in \eqref{eq:bb-projective-flatness} by a smooth partition
of unity. Multiplication by $\nu$ gives compatible scalar connection
forms on $E_\nu$, with curvature $\nu\Omega I_{s_\nu}$ for one
global closed two-form $\Omega$. With Chern--Weil normalisation,
its class $\xi$ satisfies
\begin{equation}\label{eq:bb-block-chern}
\operatorname{ch}(E_\nu)=s_\nu e^{\nu\xi},\qquad
c_1(E_\nu)=s_\nu\nu\xi.
\end{equation}
Thus \eqref{eq:bb-common-class} holds and $\xi$ is rational.

Let $\pi:E\to B'$ be the bundle projection. The global action
$a\cdot w=a^D w$ is tangent to $\cF_E$.
Lemma~\ref{lem:bb-torus-suspension} applies to its normal resolution
$0\to\cO\xrightarrow{Y}\pi^*E\to N\cF_E\to0$.
The equivariant Thom inverse of the distinguished class is a
polynomial $b_\varphi(u)\in H^*(B';\C)[u]$, where
$u\in H^2(B\C^*;\Z)$ is the weight-one class; its ordinary value
is $b_\varphi(0)$. The absolute image of the Thom class gives
\begin{equation}\label{eq:bb-euler-identity}
c_r^{\C^*}(E)b_\varphi(u)=\varphi_{\C^*}(E-\cO).
\end{equation}
By \eqref{eq:bb-block-chern}, the two characteristic expressions are
\[
c_r^{\C^*}(E)=(n_1\cdots n_r)(u+\xi)^r,\qquad
\varphi_{\C^*}(E-\cO)
=\varphi(n_1,\ldots,n_r,0,\ldots,0)(u+\xi)^d.
\]
The Euler polynomial has nonzero scalar leading coefficient, so
multiplication by it is injective over $H^*(B';\C)$.
Consequently \eqref{eq:bb-euler-identity} has the unique solution
$b_\varphi(u)=\varphi(n,0)(u+\xi)^{d-r}/\prod_i n_i$.
Setting $u=0$ proves \eqref{eq:bb-normal-formula} in every stated
degree, including $d=r$ and $d=n$.
\end{proof}

\subsection{Rationality and nonlinear transverse generators}

\begin{theorem}\label{thm:bb-rationality}
Assume \eqref{eq:bb-product-type} and \eqref{eq:bb-full-rank}.
There is a connected finite \'etale neighbourhood covering
$p:U'\to U$, of
degree $\delta\le2$, for which the distinguished localisation of
$p^*\cF$ agrees with that of its normal model under Thom
identification. With $B'=p^{-1}B$ and $\xi$ as in
\eqref{eq:bb-common-class},
\begin{equation}\label{eq:bb-general-descent-formula}
\operatorname{Res}_{\varphi}(\cF,B)
=\frac{\varphi(n_1,\ldots,n_r,0,\ldots,0)}
{\delta\,n_1\cdots n_r}\,
(p|_{B'})_*\bigl(\xi^{d-r}\cap[B']\bigr),
\quad r\le d\le n.
\end{equation}
For rational $\varphi$, the displayed formula defines a rational class.
Consequently, the Baum--Bott rationality conjecture holds for these
singular components in every degree $q+1<d<n$.
\end{theorem}
\begin{proof}
Proposition~\ref{prop:bb-transition-reduction} gives the covering
and the compatible normal bundle. The comparison is
Lemma~\ref{lem:bb-normal-comparison}, and the computation is
Lemma~\ref{lem:bb-model-formula}. For the finite \'etale map the
normal resolution remains exact under pullback. Thus
\cite[Corollary~5.3]{CorreaSuwa} gives
\[
(p|_{B'})_*\operatorname{Res}_{\varphi}(p^*\cF,B')
=\delta\operatorname{Res}_{\varphi}(\cF,B).
\]
This proves \eqref{eq:bb-general-descent-formula}. Rationality
follows from \eqref{eq:bb-common-class} and rationality of proper
push-forward.
\end{proof}

\begin{corollary}
\label{cor:bb-nonlinear-generators}
Under the hypotheses of Theorem~\ref{thm:bb-rationality}, suppose
$X_a=Y+R_a$ are convergent normal forms with $[Y,R_a]=0$.
The residue formula \eqref{eq:bb-general-descent-formula} holds
for arbitrary induced derivations $\overline R_a$ on $Q$, including
when nonlinear terms lower the local first-integral rank.
\end{corollary}
\begin{proof}
The residue formula depends only on the hypotheses of
Theorem~\ref{thm:bb-rationality}, not on the induced derivations.
By Theorem~\ref{thm:nonlinear}, the latter determine the common
first-integral algebra as
$q_0^*(\bigcap_a\ker\overline R_a)$.
\end{proof}

\begin{example}
\label{ex:bb-no-retraction}
On $C=\mathbb P^1$, put $L=\cO_C(1)$ and
$Z=\operatorname{Tot}(L^{\otimes2}\oplus L^\vee)$.
The foliation $\cG$ generated locally by $\partial_z$ and
$Y=2w_1\partial_{w_1}-w_2\partial_{w_2}$ has the global envelope
coordinate $f=w_1w_2^2$. Choose a smooth proper elliptic family
$\mathscr E\to\Delta$ with central fibre $E_0$ and nonconstant
$j$-invariant, for example the Legendre family with parameter
$\lambda_0+s$, $\lambda_0\notin\{0,1\}$, on a small disc.
Over $f^{-1}\Delta$, form
\[
X=Z\times_\Delta\mathscr E,\qquad p:X\to Z,
\qquad \cF=p^*\cG.
\]
The map $p$ is a proper holomorphic submersion, $\dim X=4$, and
$B=E_0\times\mathbb P^1$ is its compact singular component.
Submersion charts give \eqref{eq:bb-product-type}; the normal bundle is
\[
N_{B/X}=\operatorname{pr}_2^*\cO_{\mathbb P^1}(2)
\oplus\operatorname{pr}_2^*\cO_{\mathbb P^1}(-1).
\]
Thus $\xi=\operatorname{pr}_2^*c_1(L)$ and
Theorem~\ref{thm:bb-rationality} gives
\begin{equation}\label{eq:bb-nonzero-example}
\operatorname{Res}_{c_1^3}(\cF,B)
=-\tfrac12[E_0\times\{\mathrm{pt}\}]
\ne0\quad\text{in }H_2(B;\Q).
\end{equation}
Here $\varphi(2,-1,0,\ldots)=1$, $\prod n_i=-2$, and
$q+1<d<n$ since $(q,d,n)=(1,3,4)$.
The same value follows from the classical bundle formula and
submersion functoriality \cite[Theorem~1.1]{CorreaSuwa}.

A holomorphic retraction onto $B$, followed by $B\to E_0$,
would induce degree-one maps $\mathscr E_{f(x)}\to E_0$ for all
$x$ close to the zero section. Indeed properness puts each entire
nearby elliptic fibre in the retraction domain, and degree is
locally constant, equal to one on the zero section. These maps
would be isomorphisms. Since the nonconstant holomorphic function
$f$ maps a neighbourhood onto a neighbourhood of $0$, this
contradicts variation of $j$. Hence no such retraction exists.
\end{example}

\subsection{Affine actions and intersecting singular strata}
\label{subsec:bb-matroids}

Affine diagonal actions allow intersecting singular strata and
generally have noncompact singular sets. Reconstruction of rational weights relates
the following result to the original complex vector-group actions.

\par\smallskip
\noindent\textbf{Lemma.}\ \emph{Let $a_1,\ldots,a_N\in\Z^k$ span $\Q^k$,
put $q=N-k$, and set
$$
\mathcal K=\{J\subseteq\{1,\ldots,N\}:
\Span_{\C}\{a_i:i\notin J\}=\C^k\}.
$$
Every maximal face has cardinality $q$. Let
$R=\Q[v_1,\ldots,v_N]$ and let $I_{\mathcal K}$ be generated by
$\prod_{j\in J}v_j$ for $J\notin\mathcal K$.
Choose a basis $b_1,\ldots,b_q$ of the kernel of
$\Q^N\to\Q^k$, $e_i\mapsto a_i$, and put
$\ell_s=\sum_i b_{si}v_i$. Then
$\Q[\mathcal K]=R/I_{\mathcal K}$ is finite free over
$\Q[\ell_1,\ldots,\ell_q]$. In particular it is Cohen--Macaulay,
and $\ell_1,\ldots,\ell_q$ is a regular sequence.
Every minimal nonface has cardinality at most $q+1$.}
\begin{proof}
A maximal face is the complement of a minimal spanning set of
weights, hence of a basis, and so has cardinality $q$. If $J$ is a
minimal nonface, then $J\setminus\{j\}$ is a face for every
$j\in J$. Since every face has at most $q$ vertices, this gives
$|J|\le q+1$.

Order the maximal faces lexicographically by their increasing vertex lists. For $F\prec H$, let
$i=\min(F\mathbin\triangle H)\in F\setminus H$.
In the basis indexed by $H^c$, some $a_j$ with $j\in H\setminus F$
has a nonzero $a_i$-coefficient: otherwise the weights indexed by
$F^c$ would not span. Replacing $a_i$ by $a_j$ in that basis shows that
$H'=(H\setminus\{j\})\cup\{i\}$ is a maximal face. Since $i<j$,
$$
H'\prec H,\qquad F\cap H\subseteq H'\cap H=H\setminus\{j\}.
$$
Thus the intersection of each simplex with the preceding ones is a
union of its codimension-one faces.

For every maximal face $F$, projection of the relation space onto
$\Q^F$ is an isomorphism: its kernel is a relation supported on the
basis $F^c$, and both spaces have dimension $q$. Consequently the
restrictions of the $\ell_s$ to $\Q[v_j:j\in F]$ are independent
coordinates. In particular $T=\Q[\ell_1,\ldots,\ell_q]$ is a
polynomial subring of $\Q[\mathcal K]$.

Write the ordered maximal faces as $F_1,\ldots,F_m$ and put
$\mathcal K_t=\bigcup_{s\le t}2^{F_s}$. For $t>1$, let
$A_t=\{j\in F_t:F_t\setminus\{j\}\in\mathcal K_{t-1}\}$.
The preceding exchange argument says that $A_t\ne\varnothing$ and
that the new faces are exactly the subsets of $F_t$ containing $A_t$.
Hence the restriction map $\Q[\mathcal K_t]\to\Q[\mathcal K_{t-1}]$
has kernel
$$
\left(\prod_{j\in A_t}v_j\right)\Q[v_j:j\in F_t],
$$
with variables outside $F_t$ acting by zero. This kernel and
$\Q[\mathcal K_1]=\Q[v_j:j\in F_1]$ are free of rank one over $T$,
up to grading shifts. Inductively the restriction sequences split as
$T$-modules, since their quotients are free. Thus $\Q[\mathcal K]$
is finite free over $T$. Its dimension is $q$, so the $\ell_s$ form
a regular sequence of parameters and the ring is Cohen--Macaulay.
\end{proof}

\begin{theorem}\label{thm:bb-affine-vanishing}
Let $\mathcal H$ be the saturated orbit foliation of an integral
diagonal algebraic torus action on $\C^N$, of codimension $q\ge1$.
For every characteristic polynomial of degree $q+1<d\le N$,
\begin{equation}\label{eq:bb-affine-vanishing}
\alpha_{\varphi,S}(\mathcal H)=0,
\qquad \operatorname{Res}_{\varphi}(\mathcal H,S)=0,
\qquad S=\operatorname{Sing}(\mathcal H).
\end{equation}
The coefficients of $\varphi$ may be complex. In degree $d=q+1$,
the distinguished class is rational when $\varphi$ is rational.
The conclusion includes contributions carried by intersections
of singular strata.
\end{theorem}
\begin{proof}
Let $G\simeq(\C^*)^k$ be the image of the acting torus in $(\C^*)^N$,
with spanning weights $a_1,\ldots,a_N\in\Z^k$; thus $q=N-k$.
First assume that deleting any one weight leaves a spanning set. Then
$$
S=\{z:\Span_{\C}\{a_i:z_i\ne0\}\ne\C^k\}
$$
has codimension at least two. The fundamental vector fields have
saturated image: the coefficients of a meromorphic linear combination
which is holomorphic are holomorphic wherever a maximal minor is
invertible, hence extend across $S$ by Hartogs. Thus
$0\to\cO^k\to T\C^N\to N\mathcal H\to0$ is an equivariant normal
resolution, and $S=\operatorname{Sing}(\mathcal H)$.

Put $U=\C^N\setminus S$ and use maximal compact tori in the Borel
constructions. With $T=(S^1)^N$, let
$R=H^*(BT;\Q)=\Q[v_1,\ldots,v_N]$, $\deg v_i=2$.
A point lies in $U$ precisely when its zero-coordinate set belongs to
the complex $\mathcal K$ of the lemma. Consequently
$H_T^*(U;\Q)=R/I_{\mathcal K}$
\cite[Corollary~2.18]{BuchstaberPanov}.
The map $H^2(BT;\Q)\to H^2(BG;\Q)$ sends $v_i$ to the weight $a_i$.
Choose the $\ell_s$ from a basis of its kernel as in the lemma, so that
$P=H^*(BG;\Q)=R/(\ell_1,\ldots,\ell_q)$.

The homotopy pullback $U_G\simeq U_T\times^h_{BT}BG$ has simply
connected base $BT$, and the cohomology of $U_T$, $BG$ and $BT$ is
finite-dimensional in each degree. The Eilenberg--Moore spectral
sequence therefore converges \cite{SmithEM}, with
$$
E_2^{-j,t}=\operatorname{Tor}_j^R
\bigl(R/I_{\mathcal K},P\bigr)^t.
$$
The Koszul complex on the independent linear forms $\ell_s$ resolves
$P$ over $R$. By the lemma they are also a regular sequence on
$R/I_{\mathcal K}$, so the positive Koszul homology vanishes.
Only column $j=0$ survives. The multiplicative edge map gives
\begin{equation}\label{eq:bb-matroid-ring}
H_G^*(U;\Q)\simeq R/(I_{\mathcal K},\ell_1,\ldots,\ell_q).
\end{equation}
The cohomology is concentrated in even degrees, and the restriction map
from $P$ is surjective.

Let $J_P$ be its kernel. The long exact sequence of the pair, together
with $H_G^*(\C^N;\Q)=P$, gives
$$
H_G^{\mathrm{odd}}(\C^N,U;\Q)=0,\qquad
H_G^{\mathrm{even}}(\C^N,U;\Q)\hookrightarrow P
\quad\text{with image }J_P.
$$
The even relative groups inject into $P$ because $H_G^{\mathrm{odd}}(U)=0$;
the odd relative groups vanish because $P\to H_G^*(U)$ is
surjective. In particular, the forget-support map gives an
isomorphism of graded $P$-modules
\[
H_G^{\mathrm{even}}(\C^N,U;\Q)\;\xrightarrow{\ \sim\ }\;J_P.
\]
The ideal $I_{\mathcal K}$ is generated by square-free monomials
associated with minimal nonfaces of $\mathcal K$. Each such nonface
has at most $q+1$ vertices, so their images generate $J_P$ in
polynomial degrees at most $q+1$.

Lemma~\ref{lem:bb-torus-suspension} constructs the distinguished
equivariant localisation $\alpha^G_{\varphi,S}$. Its image under
the forget-support map is $\varphi_G(T\C^N-\cO^k)\in P\otimes\C$;
its image after forgetting equivariance is the distinguished
ordinary Baum--Bott class. The injection above is crucial: it
shows that a relative lift of this absolute class is \emph{unique}
in the present situation.

If $\varphi$ has rational coefficients, its absolute equivariant
class belongs to $P$. As it also belongs to $J_P\otimes\C$ and
$J_P$ is defined over $\Q$, it belongs to $J_P$ itself.
The unique relative preimage is therefore rational. Passing to
ordinary relative cohomology proves the rationality assertion,
including the critical degree $d=q+1$.

For $\varphi$ with complex coefficients and $d>q+1$, write its
absolute equivariant class as $\sum_h p_hg_h$, where the $g_h$
are homogeneous generators of $J_P$ and the $p_h$ are homogeneous.
Since $\deg g_h\le q+1<d$, every nonzero $p_h$ has positive
polynomial degree; polynomial degrees are half the cohomological
degrees. Let $\gamma_h$ be the unique relative preimage of $g_h$.
Injectivity and $P$-linearity then give
$\alpha^G_{\varphi,S}=\sum_h p_h\gamma_h$ over $\C$.
The forgetful map from equivariant to ordinary relative cohomology
annihilates the positive-degree ideal of $H^*(BG;\C)$.
Consequently the distinguished ordinary class vanishes. Alexander
duality gives the same conclusion for its residue.

For general weights, let
$W=\{(\langle a_i,v\rangle)_i:v\in\C^k\}\subset\C^N$ and
$I=\{i:e_i\in W\}$, $J=I^c$. Then
$W=\C^I\oplus W'$, where $W'=W\cap\C^J$ is rational and defines a
diagonal algebraic torus action on $\C^J$. Deleting a weight lowers
rank exactly when its coordinate vector belongs to $W$; thus this
does not occur for $W'$.
Saturation replaces each $z_i\partial_{z_i}$, $i\in I$, by
$\partial_{z_i}$. Hence $\mathcal H$ is the product of the regular
foliation on $\C^I$ and the saturated orbit foliation $\mathcal H'$
on $\C^J$. Its tangent sheaf is exactly this product: it has the
same generic span, and its normal quotient is the pullback of the
torsion-free normal sheaf of $\mathcal H'$.
The reduced codimension is still $|J|-\dim W'=q$.
For the projection $\operatorname{pr}:\C^N\to\C^J$ we have
$S=\C^I\times S'$. In product coordinates the normal resolution
of $\mathcal H$ is the pullback of the resolution of
$\mathcal H'$, and the Bott partial connections also pull back.
The defining relative \v{C}ech--de~Rham representatives therefore
give, in every degree $d>q$,
\[
\alpha_{\varphi,S}(\mathcal H)
=\operatorname{pr}^*\alpha_{\varphi,S'}(\mathcal H').
\]
This is a direct product identity; alternatively, the hypotheses
of \cite[Theorem~1.1]{CorreaSuwa} are satisfied in the product
situation. The argument above applies to $\mathcal H'$ for
$d\le|J|$. If $d>|J|$, then Alexander duality gives
\[
H^{2d}(\C^J,\C^J\setminus S';\C)
\simeq H^{\mathrm{BM}}_{2|J|-2d}(S';\C)=0,
\]
because the homological degree is negative. If $\mathcal H'$ is
regular, its singular set is empty. Both conclusions follow for
all weights, without the temporary rank assumption.
\end{proof}

\begin{corollary}\label{cor:bb-reconstructed-vanishing}
Suppose the diagonal vector configuration of the main theory has
a coordinate support $I\ne\varnothing$ with $r_I=|I|<k$,
$\dim Q_I=n-k$ and finite envelope monodromy. Then every
distinguished class of its saturated affine orbit foliation in
degree $d>q+1$ vanishes. Critical classes are rational for rational
characteristic polynomials.
\end{corollary}
\begin{proof}
The equality $\dim Q_I=n-k$ makes the residual relation space rational;
finite monodromy makes $C_I$ rational on $M_I$. A rational extension
in \eqref{eq:reconstruct}, with a rational basis of the residual
quotient, gives rational ambient weights. They are linearly
equivalent to the original weights by Theorem~\ref{thm:reconstruction}.
Clearing denominators gives the same saturated foliation and permits
Theorem~\ref{thm:bb-affine-vanishing}.
\end{proof}

\begin{example}\label{ex:bb-matroid-nonvacuous}
Take $N=5$, $k=3$ and $a_i=(1,i,i^2)$ for $1\le i\le5$.
Every triple is independent by the Vandermonde determinant.
Thus $\mathcal K=\{J:|J|\le2\}$ is the complete graph on five
vertices, and $S$ consists of points with at most two nonzero
coordinates: it is the union of the ten coordinate two-planes.
Hochster's formula and the cohomology--$\operatorname{Tor}$ identification
\cite[Theorems~3.1 and~3.2]{BuchstaberPanov} give
$H^7(U;\Q)=\widetilde H^1(\mathcal K;\Q)=\Q^6$.
Only the full vertex set contributes in degree seven, and the graph
has first Betti number $10-5+1=6$.
The relative sequence and Alexander duality identify this with
$H^8(\C^5,U;\Q)=H_2^{\mathrm{BM}}(S;\Q)$.
By Theorem~\ref{thm:bb-affine-vanishing}, the degree-four residues vanish
in this nonzero six-dimensional target; here $q=2$ and $q+1<4<5$.
\end{example}

\subsection*{Acknowledgements}
MC is partially supported by the Universit\`a degli Studi di Bari and
is a member of INdAM-GNSAGA. JS is partially supported by UNAM-PAPIIT grant IN101424.

\subsection*{AI disclosure}
AI-assisted tools were used for language editing, checking arguments and
suggesting revisions to the exposition. All mathematical ideas, proofs and
results are the authors' own. The authors have verified the mathematical
content and take full responsibility for the manuscript.

\end{document}